\documentclass[final,1p,times]{elsarticle}

\usepackage[colorlinks=true]{hyperref}
\usepackage{bbm}
\usepackage{amsfonts}
\usepackage{mathrsfs}
\usepackage{amsthm,amsmath,amssymb,mathrsfs,amsfonts,graphicx,graphics,latexsym,exscale,cmmib57,dsfont,amscd}
\usepackage[small,nohug,heads=vee]{diagrams}
\usepackage{color,soul}
\newtheorem*{thmA}{Theorem~A}
\newtheorem*{thmB}{Theorem~B}

\swapnumbers
\newtheorem*{cor*}{Corollary}
\newtheorem*{lem*}{Lemma}
\newtheorem*{prop*}{Proposition}
\newtheorem*{d)}{\ref{R1.5}d~Lemma}
\newtheorem*{e)}{\ref{R1.5}e~Lemma}

\newtheorem*{R1.4a}{\ref{R1.4}a~Lemma}
\newtheorem*{R1.4b}{\ref{R1.4}b~Lemma}
\newtheorem*{R1.4c}{\ref{R1.4}c~Lemma}

\newtheorem*{R1.8a}{\ref{R1.8}a~Lemma}
\newtheorem*{R1.8b}{\ref{R1.8}b~Corollary}

\newtheorem*{R2.1c}{\ref{R2.1}c~Lemma}
\newtheorem*{R2.1d}{\ref{R2.1}d~Lemma}

\newtheorem*{R2.6a}{\ref{R2.6}a~Theorem}
\newtheorem*{R2.6b}{\ref{R2.6}b~Theorem}
\newtheorem*{R2.6c}{\ref{R2.6}c~Theorem}

\newtheorem*{R3.1.4'}{\ref{R3.1.4}$^\prime$~Theorem}

\newtheorem{thm}{Theorem}[section]
\newtheorem{lem}[thm]{Lemma}

\newtheorem{prop}[thm]{Proposition}

\newtheorem{sthm}{Theorem}[subsection]

\theoremstyle{definition}
\newtheorem{Sn}[thm]{Standing notation}
\newtheorem*{w}{Warning}
\newtheorem*{note*}{Note}
\newtheorem*{notes*}{Notes}
\newtheorem{rem}[thm]{Remark}
\newtheorem{rems}[thm]{Remarks}
\newtheorem*{rems*}{Remarks}
\newtheorem*{rem*}{Remark}

\newtheorem*{que*}{Question}

\newtheorem{srem}[sthm]{Remark}

\newtheorem{se}[thm]{}
\newtheorem{sse}[sthm]{}
\newtheorem{exa}[sthm]{Example}

\newtheorem*{R2.1a}{\ref{R2.1}a}
\newtheorem*{R2.1b}{\ref{R2.1}b}

\newtheorem*{R2.2a}{\ref{R2.2}a}
\newtheorem*{R2.2b}{\ref{R2.2}b}

\newcommand{\K}{\boldsymbol{K}}
\newcommand{\C}{\boldsymbol{C}}

\journal{Illinois J. Math. (March 24, 2022; accepted July 21, 2022)}

\begin{document}

\begin{frontmatter}
\title{On recurrence in zero-dimensional locally compact flow with compactly generated phase group}

\author{Xiongping Dai}
\ead{xpdai@nju.edu.cn}
\address{Department of Mathematics, Nanjing University, Nanjing 210093, People's Republic of China}

\begin{abstract}
We define recurrence for a compactly generated para-topological group $G$ acting continuously on a locally compact Hausdorff space $X$ with $\dim X=0$, and then, show that if $\overline{Gx}$ is compact for all $x\in X$, the conditions (i) this dynamics is pointwise recurrent, (ii) $X$ is a union of $G$-minimal sets, (iii) the $G$-orbit closure relation is closed in $X\times X$, and (iv) $X\ni x\mapsto \overline{Gx}\in 2^X$ is continuous, are pairwise equivalent. Consequently, if this dynamics is pointwise product recurrent, then it is pointwise regularly almost periodic and equicontinuous; moreover, a distal, compact, and non-connected $G$-flow has a non-trivial equicontinuous pointwise regularly almost periodic factor.
\end{abstract}

\end{frontmatter}
\section{Introduction}\label{sec0}

Let $(G,X)$ be a \textit{flow} with phase group $G$ and with phase space $X$, which is in the following context, unless stated otherwise:
\begin{enumerate}[\textcircled{a}]
\item[\textcircled{a}] $G$ is a ``Hausdorff para-topological group'', namely: $G$ is a multiplicative group with a Hausdorff topology under which $G\times G\xrightarrow{(s,t)\mapsto st}G$ is jointly continuous, but $G\xrightarrow{t\mapsto t^{-1}}G$ need not be continuous;
\item[\textcircled{b}] $X$ is a locally compact, Hausdorff, uniform space; and moreover,
\item[\textcircled{c}] there is a left-action of $G$ on $X$, denoted $G\times X\xrightarrow{(t,x)\mapsto tx}X$ such that
\begin{enumerate}[(i)]
\item  Continuity: $(t,x)\mapsto tx$ is jointly continuous;
\item Transformation group: $ex=x$ and $(st)x=s(tx)$ for all $x\in X$ and $s,t\in G$, where $e$ is the identity of $G$;
\item Lagrange stability: $Gx$ is relatively compact (i.e. $\overline{Gx}$ is a compact set) in $X$ for all $x\in X$.
\end{enumerate}
\end{enumerate}

If the phase space $X$ is compact itself, then $(G,X)$ is Lagrange stable.
Here a Hausdorff para-topological group need not be a topological group; see, e.g., \cite[(4.20a)]{HR}. However, if $t_n\to t$ and $t_n^{-1}\to\tau$ in $G$, then $\tau=t^{-1}$; moreover, if $G$ is a locally compact Hausdorff para-topological group, then it is a topological group by Ellis' Joint Continuity Theorem.

If $\textrm{Homeo}\,(X)$ stands for the set of self homeomorphisms of $X$ endowed with the topology of uniform convergence on compacta, then $\textrm{Homeo}\,(X)$ is a Hausdorff para-topological group such that $\textrm{Homeo}\,(X)\times X\xrightarrow{(f,x)\mapsto f(x)}X$ is a flow having no Lagrange stability in general.

\begin{se}\label{R0.1}
It is well known that every recurrent point is almost periodic (in fact periodic) for any continuous-time flow on the plane \cite{NS}. In \cite[p.~764]{G46} W.\,H.~Gottschalk suggested an interesting question that is to determine conditions under which pointwise recurrence of a flow implies pointwise almost periodicity.

The main aim of this paper is to consider the relationships of the following important dynamics of $(G,X)$:
\begin{enumerate}[(1)]
\item The pointwise ``recurrence'' in certain sense (cf.~Def.~\ref{R2.2}).
\item The pointwise almost periodicity (cf.~Def.~\ref{R1.5}a).
\item $X$ is a union of minimal sets.
\item The orbit closure relation is closed (cf.~Def.~\ref{R1.3}).
\item The orbit closure mapping is $\ulcorner$\!continuous\!$\urcorner$ on $X$ (cf.~Def.~\ref{R1.2}a).
\item The orbit closure mapping is continuous on $X$ (cf.~Def.~\ref{R1.2}b).
\item The local weak almost periodicity (cf.~Def.~\ref{R1.5}b).
\end{enumerate}

It is evident that $(6)\Rightarrow(5)\Leftrightarrow(4)\Rightarrow(3)\Leftrightarrow(2)\Rightarrow(1)$. However, even for $G=\mathbb{Z}$ or $\mathbb{R}$ and $X$ a compact metric space, that ``$(1)\Rightarrow(5)$'' is not generally true is shown by very simple examples.
On the other hand, we consider the following problem, which is motivated by Furstenberg's work~\cite{F}.
\begin{enumerate}[(8)]
\item \textit{If $(G,X)$ is a distal\,(cf.~Def.~\ref{R1.6}) compact non-minimal flow, does $(G,X)$ have an equicontinuous factor?}
\end{enumerate}
\end{se}

\begin{se}
Note that if $(G,X)$ is distal and the orbit closure mapping of $(G,X\times X)$ is continuous, then it is not difficult to prove that the regionally proximal relation $\textrm{RP}$ of $(G,X)$ is equal to $\Delta_X$ so that $(G,X)$ is equicontinuous \cite{E58} (cf.~Theorem~\ref{R3.1.1}). In McMahon and Wu \cite[(3.4)]{MW76}, there exists an example of a countable group $G$ acting on a 0-dimensional compact $X$ such that $(G,X)$ is minimal distal non-equicontinuous (also cf.~Example~\ref{R3.1.5}). So in this example, $(G,X\times X)$, with 0-dimensional compact phase space $X\times X$ and with countable (so $\sigma$-compact) phase group $G$, satisfies that $(1)\not\Rightarrow(5)$ and $(3)\not\Rightarrow(5)$.
Thus we need to impose some restrictive condition on the phase group $G$ for $(1)\textrm{ or }(3)\Rightarrow(5)$.

On the other hand, there are $\mathbb{R}$-flows and $\mathbb{Z}$-flows that are distal non-equicontinuous with compactly generated phase groups \cite{AGH, F} so that (1) or (2) $\not\Rightarrow$ (6) in \ref{R0.1}. Thus we also need impose some restrictive condition on the phase space $X$.

Let $X$ be a 0-dimensional compact space. In \cite[Theorems~5 and 6]{G46}, Gottschalk shows that $(1)\Rightarrow(5)\Leftrightarrow(7)$ for $\mathbb{Z}_+$-flow on $X$, and see also \cite[Theorem~7.12]{AD} for $\mathbb{Z}$-flow on $X$.

Recall that $G$ is \textit{generative} provided that $G$ is abelian and generated by some compact neighborhood of $e$ (cf.~\cite[Def.~6.01]{GH}). Using ``replete semigroup'', in \cite{GH} Gottschalk and Hedlund formulate a definition of recurrence. Then, with this notion, they prove that $(1) \Rightarrow (5)\Leftrightarrow(7)$ for each flow with generative phase group and with 0-dimensional phase space (cf.~\cite[Theorems~7.07 and 7.08]{GH}).

In \cite{AGW}, Auslander et al. introduce a definition of recurrence in terms of ``cone''. With their notion of recurrence, they show that $(1) \Rightarrow (5)$ for every flow $(G,X)$ with finitely generated phase group $G$ not necessarily abelian and with 0-dimensional compact metric phase space $X$ (cf.~\cite[Theorem~1.8]{AGW} and Proposition~\ref{5.5} below).

We will formulate a recurrence and extend the theorem of Auslander et al. \cite[Theorem~1.8]{AGW} to any flow $(G,X)$ with compactly generated phase group $G$ not necessarily countable or abelian, and, with 0-dimensional phase space $X$ not necessarily compact metrizable, as follows:

\begin{thmA}[see Thm.~\ref{R2.6}a for the full statement]
Let $(G,X)$ be a flow, where $G$ is compactly generated and $X$ is 0-dimensional (cf.~Def.~\ref{R1.7}). Then properties (1)\,--\,(7) in \ref{R0.1} are pairwise equivalent.
\end{thmA}

\begin{w}
The orbit closure relation being closed or orbit closure mapping $\ulcorner$\!continuous\!$\urcorner$ $\not\Rightarrow$ orbit closure mapping being upper semi-continuous in general non-compact flows (see Theorem~\ref{R2.6}c).
\end{w}

Comparing with the case $X$ compact, we have to face the essential point that a net in $X$ need not have a convergent subnet. In addition, comparing with the case $G=\mathbb{Z}_+$, for any point $x$ and any closed subset $U$ of $X$ we have no the ``first visit time'' $t\in G$ with $tx\in U$.

In addition, as an application of Theorem~A, we can conclude the following using the component relation:
\begin{enumerate}[(9)]
\item \textit{Every distal, compact, and non-connected flow has a non-trivial equicontinuous pointwise regularly almost periodic factor if the phase group is compactly generated like $\mathbb{Z}$ and $\mathbb{R}$ (see Thm.~\ref{R3.1.2} and Corollary to Thm.~\ref{R3.1.3}).} See \ref{R1.5}c for the definition of regular almost periodicity.
\end{enumerate}

If a flow $(G,X)$ is pointwise product almost periodic (i.e. $(G,X\times X)$ is pointwise almost periodic) then $(G,X)$ is distal (cf.~\cite[Theorem~1]{E58}); and moreover, if in addition $G$ is generative and $X$ is 0-dimensional, then $(G,X)$ is equicontinuous (cf.~\cite[Theorem~2]{E58}). Now we shall improve Ellis' theorems as follows:

\begin{thmB}[{see Thm.~\ref{R3.1.1} and Thm.~\ref{R3.1.3}}]
Let $(G,X)$ be any flow with $G$ compactly generated and with $\dim X=0$. If $(G,X)$ is pointwise product recurrent, then $(G,X)$ is pointwise regularly almost periodic and equicontinuous.
\end{thmB}
\end{se}

\begin{rem}
It should be mentioned that a part of Theorem~A has been proved by Reid \cite{R2} under the framework that $G$ is ``equicontinuously generated'' by a subset $S$ of $\textrm{Homeo}\,(X)$; that is, $G=\bigcup_{r\ge1} S^r$, where $\textrm{id}_X\in S=S^{-1}$, and $S$ is equicontinuous. However, our approaches are different completely with and much simpler than Reid \cite[Theorems~1.2, 1.3]{R2} of using very technical arguments of compact-open invariant sets. It turns out that Theorem~C concisely proved here implies the ``equicontinuously generated'' case of Reid~\cite{R2} (see Propositions~\ref{R5.5} and \ref{R5.6} for alternative simple proofs of Reid's results).
\end{rem}

\begin{Sn}
By \textcircled{b} there is a uniformity structure, denoted $\mathscr{U}_X$ or $\mathscr{U}$, on $X$. In the sequel, $\mathfrak{N}_x$ stands for the neighborhood filter at $x$ of $X$ and $N_G(x,U)=\{t\,|\,t\in G, tx\in U\}$ for all $x\in X$ and all $U\in\mathfrak{N}_x$.
\end{Sn}
\section{Preliminaries}\label{secR1}

\begin{se}\label{R1.1}
If $Y$ is a Hausdorff space then $2^Y$ will stand for the collection of closed subsets of $Y$. For a net $\{Y_i\}$ in $2^Y$ and $K\subseteq Y$, we say `$Y_i\rightharpoonup K$ in $2^Y$', denoted $\limsup_iY_i=K$, if
$$K=\{y\in Y\,|\,\exists\textrm{ a subnet }\{Y_{i_k}\}\textrm{ from }\{Y_i\}\textrm{ and }y_{i_k}\in Y_{i_k}\textrm{ s.t. }y_{i_k}\to y\}.$$
It turns out that in the case of $\rightharpoonup$, every net $\{K_i\}$ in $2^Y$ converges and $\limsup_iK_i\in 2^Y$. However, it is possible that $\limsup_iK_i=\emptyset$.

\begin{lem*}
If $K_i\rightharpoonup K$ in $2^Y$ and $\{K_{i_n}\}$ is any subnet of $\{K_i\}$ with $K_{i_n}\rightharpoonup K^\prime$, then $K^\prime\subseteq K$. Moreover, if $K_i$ is invariant for all $i$, then $\limsup_iK_i$ is also invariant.
\end{lem*}

If $Y$ is a discrete space and $\{Y_i\}_{i=1}^\infty$ with $Y_i\rightharpoonup K$, then $K=\bigcap_{n\ge1}\bigcup_{i\ge n}Y_i=\limsup_{i\to\infty}Y_i$; in other words, $K=\{Y_i\ \textrm{i.o.}\}$. We will concern with $Y=G$ or $X$ for a flow $(G,X)$.
\end{se}

\begin{se}\label{R1.2}
Associated to $(G,X)$ we define the `orbit closure mapping' $\mathscr{O}_G\colon X\rightarrow 2^X$ by $x\mapsto \overline{Gx}$. Then we say that:
\begin{enumerate}[a.]
\item $\mathscr{O}_G$ is \textit{$\ulcorner$\!continuous\!$\urcorner$}, provided that if $x_i\to x$ in $X$, then $\overline{Gx_i}\rightharpoonup\overline{Gx}$ in $2^X$.
\item $\mathscr{O}_G$ is \textit{upper semi-continuous}, provided that for every $x\in X$ and every neighborhood $U$ of $\overline{Gx}$ there is a $V\in\mathfrak{N}_x$ such that $\overline{Gy}\subseteq U$ for all $y\in V$. Further, $\mathscr{O}_G$ is called \textit{continuous} if $x_i\to x$ in $X$ implies that $\overline{Gx_i}\to\overline{Gx}$ in $2^X$ with the Hausdorff topology.
\end{enumerate}

\begin{lem*}
If $\mathscr{O}_G$ is upper semi-continuous, then $\mathscr{O}_G$ is $\ulcorner$\!continuous\!$\urcorner$, and moreover, $\mathscr{O}_G$ is continuous.
\end{lem*}

\begin{proof}
If $x_i\to x$ in $X$ and $\overline{Gx_i}\rightharpoonup L$ in $2^X$, then $\overline{Gx}\subseteq L$ for $L$ is closed $G$-invariant and $x$ is in $L$. On the other hand, if we additionally assume $\mathscr{O}_G$ is upper semi-continuous, then for every neighborhood $V$ of $\overline{Gx}$ we have $\overline{Gx_i}\subseteq V$ eventually. Since $\overline{Gx}$ is compact and $X$ is Hausdorff, $\overline{Gx_i}$ is eventually separated from any point outside $\overline{Gx}$, so $L\subseteq\overline{Gx}$. Thus $\mathscr{O}_G$ is $\ulcorner$\!continuous\!$\urcorner$.

On the other hand, since the phase mapping $(t,x)\mapsto tx$ of $G\times X$ onto $X$ is jointly continuous, $\mathscr{O}_G$ is automatically lower semi-continuous (i.e. $\forall x^\prime\in\overline{Gx}$ $\exists x_i^\prime\in\overline{Gx_i}$ s.t. $x_i^\prime\to x^\prime$ for $x_i\to x$), $\mathscr{O}_G$ is continuous. The proof is completed.
\end{proof}

Given $U\subset X$, define invariant sets
\begin{enumerate}[c.]
\item $U_G^*=\{x\in X\,|\,\overline{Gx}\subseteq U\}$ and $U_G^\infty=\bigcap_{g\in G}gU$.
\end{enumerate}
If $U$ is closed, then $U_G^*=U_G^\infty$ is closed. However, $U_G^*\varsubsetneq U_G^\infty$ in general.
\begin{enumerate}[d.]
\item $(G,X)$ is \textit{equicontinuous at a point} $x\in X$ if given $\varepsilon\in\mathscr{U}$ there is $U\in\mathfrak{N}_x$ such that $gU\subseteq\varepsilon[gx]$ for all $g\in G$. $(G,X)$ is \textit{equicontinuous} iff $(G,X)$ is pointwise equicontinuous.
\end{enumerate}
Clearly, if $(G,X)$ is equicontinuous at $x_0\in X$ then $\mathscr{O}_G\colon X\rightarrow 2^X$ is continuous at the point $x_0$. So if $(G,X)$ is equicontinuous, then $\mathscr{O}_G$ is continuous on $X$.

Note that equicontinuity is independent of the topology of $G$, but it depends on the specific uniform structure $\mathscr{U}$ rather than merely on the topology of $X$.
\end{se}

\begin{se}\label{R1.3}
The `orbit-closure relation $R_o(X)$ of $(G,X)$' is defined by $(x,y)\in R_o(X)$ iff $y\in\overline{Gx}$.
\begin{enumerate}[a.]
\item \textit{If $R_o(X)$ is closed, then $R_o(X)$ is symmetric, i.e., $R_o=R_o^{-1}$.}
\end{enumerate}
\begin{proof}
This is because for $(x,y)\in R_o(X)$ there is a net $t_n\in G$ with $t_nx\to y$ and $(t_nx,x)\in R_o(X)$ so that $(t_nx,x)\to(y,x)\in R_o(X)$.
\end{proof}

\begin{enumerate}[b.]
\item \textit{$\mathscr{O}_G$ is $\ulcorner$\!continuous\!$\urcorner$ on $X$ if and only if $R_o(X)$ is closed in $X\times X$.} (Note: Lagrange stability is surplus here.)
\end{enumerate}
\begin{proof}
Indeed, suppose $\mathscr{O}_G$ is $\ulcorner$\!continuous\!$\urcorner$, and, let $(x_i,y_i)\in R_o(X)\to(x,y)$. Then $\overline{Gx_i}\rightharpoonup\overline{Gx}$ by \ref{R1.2}a) and $y_i\in\overline{Gx_i}$ with $y_i\to y$. So $y\in\overline{Gx}$ by \ref{R1.1} and $(x,y)\in R_o(X)$. Thus $R_o(X)$ is closed. Conversely, assume $R_o(X)$ is closed, and, let $x_i\to x$ in $X$ and $\overline{Gx_i}\rightharpoonup\mathfrak{O}$ in $2^X$. If $\overline{Gx}\not=\mathfrak{O}$, then $\overline{Gx}\varsubsetneq\mathfrak{O}$. For $y\in\mathfrak{O}\setminus\overline{Gx}$ there is a net $y_{i_k}\in \overline{Gx_{i_k}}$ with $y_{i_k}\to y$. Since $(x_{i_k},y_{i_k})\in R_o(X)$, so $(x,y)\in R_o(X)$ and $y\in \overline{Gx}$ contrary to $y\in\mathfrak{O}\setminus\overline{Gx}$. Thus $\overline{Gx}=\mathfrak{O}$ and $\mathscr{O}_G$ is $\ulcorner$\!continuous\!$\urcorner$.
\end{proof}

\end{se}

\begin{se}\label{R1.4}
We say that a set $K$ in $X$ is `minimal' under $(G,X)$ if $\overline{Gx}=K$ for all $x\in K$. In this case $K$ is a closed invariant subset of $X$, and moreover, $K$ is compact by (iii) in \textcircled{c}. So every minimal flow is compact in our setting.

\begin{R1.4a}
Let $R_o[y]=\{x\,|\,x\in X, (x,y)\in R_o(X)\}$ for $(G,X)$ and $y\in X$. If $R_o[y]$ is closed, then $\overline{Gy}$ is minimal. Here $(G,X)$ need not be Lagrange stable.
\end{R1.4a}
\begin{note*}
Here $\{x\in X\,|\,(y,x)\in R_o(X)\}$ is always closed but it gives no minimality.
\end{note*}

\begin{proof}
Let $x\in\overline{Gy}$. There is a net $t_n\in G$ with $t_ny\to x$. Since $(t_ny,y)\in R_o(X)$, so $(x,y)\in R_o(X)$ and $y\in\overline{Gx}$. Thus $\overline{Gy}=\overline{Gx}$ is minimal for $x\in\overline{Gy}$ is arbitrary.
\end{proof}

Clearly $R_o(X)$ is symmetric iff $\overline{Gx}$ is minimal for all $x\in X$. Moreover, if $R_o(X)$ is closed, then $R_o(X)$ is an invariant closed equivalence relation on $X$, and, $X/G:=X/R_o(X)$ with the quotient topology is locally compact Hausdorff by Lemma~\ref{R1.4}b below.

\begin{R1.4b}
If $R_o(X)$ is symmetric for $(G,X)$, then the quotient mapping $\rho\colon X\rightarrow X/G$ is open.
\end{R1.4b}
\begin{note*}
$X/G$ need not be Hausdorff if $R_o(X)$ is not closed.
\end{note*}

\begin{proof}
Let $U$ be an open subset of $X$, $U\not=\emptyset$. Then $\rho^{-1}\rho U=GU$. Indeed, for each $x\in U$, choose a compact $V\in\mathfrak{N}_x$ with $V\subseteq U$. Since $x$ is almost periodic and $N_G(x,V)$ is discretely syndetic (cf.~\ref{R1.5}a below), there is a finite set $K\subset G$ such that $Gx\subseteq KV$ and $\overline{Gx}\subseteq K\overline{V}\subseteq GU$. By $\rho^{-1}\rho U=\bigcup_{x\in U}\overline{Gx}$, it follows that $\rho^{-1}\rho U=GU$ is open in $X$. Thus $\rho U$ is open in $X/G$. The proof is complete.
\end{proof}

\begin{R1.4c}[{cf.~\cite[Thm.~1]{G56} for $X$ compact}]
The orbit-closure relation $R_o(X\times X)$ of $(G,X\times X)$ is closed if and only if $(G,X)$ is equicontinuous.
\end{R1.4c}

\begin{proof}
Sufficiency is obvious from \ref{R1.2}d and that if $(G,X)$ is equicontinuous, then so is $(G,X\times X)$.
Conversely, suppose $R_o(X\times X)$ is closed.
Let $x_0\in X$. To show that $(G,X)$ is equicontinuous at $x_0$, suppose the contrary that $(G,X)$ were not equicontinuous at $x_0$. Then there would be nets $(x_n,x_n^\prime)\in X\times X$, $t_n\in G$ and an open index $\varepsilon\in\mathscr{U}$ such that $(x_n,x_n^\prime)\to(x_0,x_0)$ and $t_n(x_n,x_n^\prime)\notin\varepsilon$. By $\overline{G(x_n,x_n^\prime)}=\overline{Gt_n(x_n,x_n^\prime)}$, it follows that
$$
\varepsilon\supseteq\Delta_X\supseteq\overline{G(x_0,x_0)}={\lim}_n\overline{G(x_n,x_n^\prime)}={\lim}_n\overline{Gt_n(x_n,x_n^\prime)}\nsubseteq\varepsilon
$$
which is a contradiction, where $\Delta_X$ is the diagonal of $X\times X$. This contradictions shows that $(G,X)$ is equicontinuous.
\end{proof}

In particular, if $X$ is compact and the relation $R_o(X\times X)$ of $(G,X\times X)$ is closed, then $(G,X)$ is uniformly equicontinuous.
\end{se}

\begin{se}\label{R1.5}
A subset $A$ of $G$ is `syndetic' in $G$ if there exists a compact subset $K$ of $G$ with $G=K^{-1}A$. A set $B$ is said to be `thick' in $G$ if for all compact set $K$ in $G$ there corresponds an element $t\in B$ such that $Kt\subseteq B$. It is a well-known fact that
\begin{enumerate}[$\bullet$]
\item $A$ is syndetic in $G$ iff it intersects non-voidly with every thick subset of $G$.
\end{enumerate}
Although $G$ need not be discrete here, yet if the syndetic/thick is defined in the sense of discrete topology of $G$ then it will be called `discretely syndetic/thick'. For instance, in $\mathbb{R}$ with the usual euclidean topology, $\mathbb{Z}$ is a discrete subgroup and $\mathbb{Z}$ is syndetic; however, $\mathbb{Z}$ is not discretely syndetic in $\mathbb{R}$.

\begin{enumerate}[a.]
\item A point $x\in X$ is `almost periodic' (a.p) under $(G,X)$ if $N_G(x,U)$ is a syndetic subset of $G$ for all $U\in\mathfrak{N}_x$.
Since $\overline{Gx}$ is compact by (iii) in \textcircled{c} here, $x\in X$ is a.p if and only if $\overline{Gx}$ is minimal under $(G,X)$ (cf., e.g.~\cite[Note~2.5.1]{AD}). Thus
\begin{enumerate}[$\bullet$]
\item $N_G(x,U)$ is discretely syndetic in $G$ for every a.p point $x$ of $(G,X)$. In addition, $x$ is a.p iff given $U\in\mathfrak{N}_x$ there is a finite set $K$ in $G$ such that $Gx\subseteq KU$.
\end{enumerate}

\item Following \cite{GH, G56}, we say that:
\begin{enumerate}[1)]
\item $(G,X)$ is `weakly a.p' if given $\alpha\in\mathscr{U}$ there is a compact subset $F$ of $G$ such that $Ftx\cap\alpha[x]\not=\emptyset$ for all $x\in X$ and $t\in G$.
\item $(G,X)$ is `locally weakly a.p' if given $\alpha\in\mathscr{U}$ and $x\in X$ there is a compact subset $F$ of $G$ and a $V\in\mathfrak{N}_x$ such that $Fty\cap\alpha[y]\not=\emptyset$ for all $y\in V$ and all $t\in G$. Equivalently, $(G,X)$ is locally weakly a.p iff for each $x\in X$ and all $U\in\mathfrak{N}_x$ there is a compact set $F$ in $G$ and a set $V\in\mathfrak{N}_x$ such that $GV\subseteq F^{-1}U$.
\end{enumerate}
Clearly, if $X$ is compact, then $(G,X)$ is weakly a.p iff $(G,X)$ is locally weakly a.p; if $(G,X)$ is locally weakly a.p, it is pointwise a.p under $(G,X)$. Moreover, if $(G,X)$ is locally weakly a.p, then $\mathscr{O}_G\colon X\rightarrow 2^X$ is upper semi-continuous (see Lemma~\ref{R2.4}).

\item (see \cite{GH, E58, MR} for $G$ an abelian group).
\begin{enumerate}[1)]
\item A point $x$ is said to be \textit{regularly a.p} under $(G,X)$, denoted $x\in P_{r.a.p}(X)$, if $N_G(x,U)$ contains a syndetic normal subgroup of $G$ for every $U\in\mathfrak{N}_x$; if every point of $X$ is regularly a.p under $(G,X)$, i.e., $P_{r.a.p}(X)=X$, then $(G,X)$ is called \textit{pointwise regularly a.p}.

\item We say $(G,X)$ is (uniformly) \textit{regularly a.p} if every $\alpha\in\mathscr{U}$ there is a syndetic normal subgroup $A$ of $G$ such that $Ax\subseteq\alpha[x]$ for all $x\in X$.

\item We call $(G,X)$ a \textit{point-regularly a.p flow} if there is a regularly a.p point that has a dense orbit in $X$.
\end{enumerate}
\end{enumerate}

See \cite[Theorem~12.55]{GH} for an example of point-regularly a.p $\mathbb{Z}$-flows that is not equicontinuous. Note that even for $X$ a compact metric space, a pointwise regularly a.p (in fact, pointwise periodic) $\mathbb{Z}$-flow need not be regularly a.p (see Example~\ref{R3.3.1}).

\begin{d)}[{cf.~\cite[Lem.~5.04]{GH}}]
Let $T$ be a discrete group and let $A$, $B_1,\dotsc,B_n$ be syndetic subgroups of $T$. Then:
\begin{enumerate}[(1)]
\item $\bigcap_{i=1}^n B_i$ is a syndetic subgroup of $T$.
\item There exists a syndetic normal subgroup $H$ of $T$ such that $H\leq A$. In fact, we can take $H=N(A)=\bigcap_{t\in T}t^{-1}At$.
\end{enumerate}
\end{d)}

\begin{proof}
(1). As $T/B_1, \dotsc, T/B_n$ are finite, it follows from $t(B_1\cap\dotsm\cap B_n)=tB_1\cap\dotsm\cap tB_n\ \forall t\in T$ that $T/\left(\bigcap_{i=1}^nB_i\right)$ is finite, whence $\bigcap_{i=1}^n B_i$ is a syndetic subgroup of $T$.

(2). Choose a finite set $K$ in $T$ such that $KA=T$. Then $N(A)=\bigcap_{t\in K}tAt^{-1}$, which is normal and is syndetic by (1). The proof is complete.
\end{proof}

Comparing with \cite{GH} where Gottschalk and Hedlund is by using the permutation group of $T/A$ for proving (2), the point of our alternative proof is to choose $H=N(A)$ here.

\begin{e)}[{cf.~\cite[Prop.~2.8]{E69} and \cite[Thm.~1.13]{A88} for $G$ a topological group instead of $G$ a para-topological group}]
Let $(G,X)$ be a flow and $A$ a syndetic normal subgroup of $G$. Then $(G,X)$ is pointwise a.p iff $(A,X)$ is pointwise a.p.
\end{e)}

\begin{proof}
Sufficiency is obvious (without using the assumption that $A$ is normal). Conversely, suppose $(G,X)$ is pointwise a.p and let $x\in X$. Let $y\in\overline{Ax}$ be an a.p point under $(A,X)$. Take $t_n\in G$ with $t_ny\to x$. Let $K$ be a compact subset of $G$ with $G=K^{-1}A$. Then there are $k_n\in K$ with $k_nt_n\in A$ and $k_n\to k\in K$. So $k_nt_ny\to kx\in\overline{Ay}$. Thus $kx$ is a.p under $(A,X)$. Since $A$ is normal, it is not difficult to prove that $x$ is also a.p for $(A,X)$. The proof is complete.
\end{proof}

\begin{cor*}
Let $(G,X)$ be a flow with $G$ a discrete group and $A$ a syndetic subgroup of $G$. Then $(G,X)$ is pointwise a.p iff $(A,X)$ is pointwise a.p.
\end{cor*}

\begin{proof}
By Lemma~\ref{R1.5}d,
there is a syndetic normal subgroup $H$ of $G$ such that $H\le A$. Then by Lemma~d) above and $H\triangleleft A$, $(G,X)$ is pointwise a.p iff $(H,X)$ is pointwise a.p iff $(A,X)$ is pointwise a.p. The proof is complete.
\end{proof}

\begin{note*}
If $(G,X)$ is minimal in addition, then $\mathscr{O}_A\colon X\rightarrow 2^X$ is continuous; see Theorem~\ref{R2.7}. However, there is no an analogous inheritance in general for the regular almost periodicity of non-discrete group actions.
\end{note*}
\end{se}

\begin{se}\label{R1.6}
We say that $(G,X)$ is \textit{distal} iff we have for all $x,y\in X$ with $x\not=y$ that $\Delta_X\cap\overline{G(x,y)}=\emptyset$. Clearly, distality is independent of the topology of $G$. Moreover, a distal flow is pointwise a.p.
\end{se}

\begin{se}\label{R1.7}
A Hausdorff space is said to be \textit{0-dimensional} if it has a base consisting of clopen sets. For example: 1. A locally compact, Hausdorff, and totally disconnected topological space is 0-dimensional (\cite[Theorem~3.5]{HR}); 2. Ellis' ``two-circle minimal set'' is compact Hausdorff 0-dimensional (\cite[Example~5.29]{E69}).

$G$ is called \textit{compactly generated} if there exists a compact subset $\Gamma$ of $G$ with $e\in \Gamma=\Gamma^{-1}$, called a \textit{generating set}, such that
$G=\bigcup_{r=1}^\infty\Gamma^r$. In this case, $G$ is $\sigma$-compact, not necessarily locally compact. Moreover, the generating set $\Gamma$ is not unique; e.g., $t^{-1}\Gamma t$, for $t\in G$, is so. In addition, $\Gamma$ need not be a neighborhood of $e$. For instance, $\Gamma=[-1,1]\times\{0\}\cup\{0\}\times[-1,1]$ is a compact generating set of $(\mathbb{R}^2,+)$ but it is not a neighborhood of $o=(0,0)\in\mathbb{R}^2$.
\end{se}

\begin{se}[inheritance]\label{R1.8}
If $G$ is compactly generated and $S$ a closed syndetic subgroup of $G$, is $S$ compactly generated too? Towards a positive solution we need the following lemma, which is a variation of \cite[Theorem~6.10]{GH}.

\begin{R1.8a}
Let $T$ be a group, let $S$ be a subgroup of $T$, $\Gamma\subseteq T$ a set with $e\in\Gamma=\Gamma^{-1}$ such that $T=\Gamma S$. Define $\Psi=\Gamma^3\cap S$. Then $\Gamma^n\cap S\subseteq\Psi^n$ for all $n\in\mathbb{Z}$.
\end{R1.8a}

\begin{proof}
Let $n$ be any positive integer, and, let $\gamma_1,\dotsc,\gamma_n\in \Gamma$ such that $s_n:=\gamma_n\dotsm \gamma_1\in S$; that is, $s_n\in \Gamma^n\cap S$.
By $T=\Gamma S=S\Gamma$, it follows that, for $1\le i<n$, there exists $s_i\in S$ with $s_i\in \Gamma\gamma_i\dotsm \gamma_1$. Clearly, $s_n\in \Gamma\gamma_n\dotsm \gamma_1$. Now $s_{i+1}s_i^{-1}\in \Gamma\gamma_{i+1}\Gamma\subseteq \Gamma^3$ and $s_{i+1}s_i^{-1}\in S$ for all $1\le i<n$. So $s_{i+1}s_i^{-1}\in \Psi$ and $s_{i+1}\in \Psi s_i$ for $1\le i<n$. Also $s_i\in \Gamma\gamma_1\subseteq \Gamma^3$ and $s_1\in S$, thus $s_1\in \Psi$. Then $s_n\in \Psi s_{n-1}\subseteq \Psi^2s_{n-2}\subseteq\dotsm\subseteq \Psi^{n-1}s_1\subseteq \Psi^n$. The proof is completed.
\end{proof}

\begin{R1.8b}
Let $G$ be a compactly generated topological group and $S$ a syndetic closed subgroup of $G$. Then $S$ is also compactly generated.
\end{R1.8b}

\begin{proof}
Let $G$ be compactly generated by $\Gamma$ with $e\in\Gamma=\Gamma^{-1}$. Since $S$ is syndetic, we can take a compact symmetric set $K\subset G$ with $G=KS$. Define $\Psi=(\Gamma\cup K)^3\cap S$, which is compact in $S$ such that $e\in\Psi=\Psi^{-1}$. By Lemma~\ref{R1.8}a, $S=\bigcup_{r\ge1}\Psi^r$. Thus $S$ is compactly generated by $\Psi$.
The proof is completed.
\end{proof}
\end{se}

\begin{se}[a note to para-topological group]\label{R1.9}
Let $(G,X)$ be a flow. Let $K\subset G$ be a compact set. Then $K^{-1}$ need to be compact subset of $G$. However, we can assert that $K^{-1}W$ is compact for every compact set $W$ in $X$ for $(G,X)$:

\begin{lem*}
If $K\subset G$ and $W\subset X$ are compact sets. Then $K^{-1}W$ is a compact subset of $X$.
\end{lem*}

\begin{proof}
We equip $\textrm{Homeo}\,(X)$ with the topology of uniform convergence on compacta. Then the mapping $G\rightarrow\textrm{Homeo}\,(X)$ defined by $t\mapsto t_X$, where $t_X\colon x\in X\mapsto tx\in X$, is continuous so $\mathbb{K}=\{t_X\,|\,t\in K\}$ is a compact subset of $\textrm{Homeo}\,(X)$. Let $t_n\to t$ in $K$ and $x\in X$. Since $\overline{Gx}$ is compact,
we may assume (a subnet of) $(t_{nX})^{-1}x\to x^*$ so that $x=t_{nX}(t_{nX})^{-1}x\to t_Xx^*$. Thus $x=t_Xx^*$ and $(t_X)^{-1}x=x^*$. This implies that $t_{nX}^{-1}\to (t_X)^{-1}$. Thus $K\rightarrow\textrm{Homeo}\,(X)$ defined by $t\mapsto (t_X)^{-1}$ is continuous. Since $t_X(t^{-1})_Xx=tt^{-1}x=x$ for all $t\in G$ and $x\in X$, so $(t_X)^{-1}=(t^{-1})_X$ for all $t\in G$. Then $\mathbb{K}^{-1}$ is compact in $\textrm{Homeo}\,(X)$. Therefore, $K^{-1}W=\mathbb{K}^{-1}W$ is compact in $X$. The proof is completed.
\end{proof}
\end{se}
\section{Pointwise recurrence in 0-dimensional flows}\label{secR2}

This section will be devoted to proving the full statement of Theorem~A stated in $\S$\ref{sec0}, which implies that for some class of flows, the pointwise recurrence is equivalent to the pointwise almost periodicity.

\begin{se}\label{R2.1}
Let $G$ be compactly generated by $\Gamma$ such that $e\in\Gamma=\Gamma^{-1}$ as in Def.~\ref{R1.7}. For example, if $G$ is finitely generated or if $G$ is generative, then $G$ is compactly generated.

\begin{R2.1a}[word length]
For any $g\in G$ with $g\not=e$, the \textit{$\Gamma$-length of} $g$, denoted $|g|$, is the smallest integer $r$ such that $g=\gamma_1\dotsm\gamma_r$ with $\gamma_i\in\Gamma$ for $1\le i\le r$, and, write $\K(g)=\Gamma^{|g|-1}\cdot g$, where $\Gamma^0=\{e\}$ and $\K(g)$ is a compact subset of $G$. Clearly, $e\notin\K(g)$ for $g\in G\setminus\{e\}$ with $|g|\ge1$.
\end{R2.1a}

\begin{R2.1b}[cone]
A subset $\C$ of $G$ is called a \textit{$\Gamma$-cone} in $G$ if there exists a net $\{g_i\,|\,i\in\Lambda\}$ in $G$ such that $|g_i|\nearrow\infty$ and $\K(g_i)\rightharpoonup\C$ in $2^G$. Here $|g_i|\nearrow\infty$ means that if $i\le i^\prime$ in the ordering of $\Lambda$ then $|g_i|\le|g_{i^\prime}|$.
Clearly, every $\Gamma$-cone in $G$ is a closed set in $G$.

Moreover, if $G$ is discrete (i.e. finitely generated), then $e\notin\C$. However, this is not the case in non-discrete $G$. For instance, let $G=(\mathbb{R},+)$ with the usual euclidean topology and let $\Gamma=[-1,1]$; then for $g_i=i+1/i$, $i=1,2,\dotsc$, we have that $0\notin\K(g_i)\rightharpoonup\C=[0,\infty)$ with $0\in\C$.
\end{R2.1b}

\begin{R2.1c}[{cf.~\cite[Prop.~1.2 and Prop.~1.5]{AGW} for $G$ finitely generated}]
Let $G$ be compactly generated by $\Gamma$. Then:
\begin{enumerate}[$(1)$]
\item If $\{g_i\}$ is a net in $G$ with $|g_i|\nearrow\infty$, then $\K(g_i)\rightharpoonup\C\not=\emptyset$.
\item If $\C$ is a $\Gamma$-cone in $G$, then $\C$ is a discretely thick subset of $G$.
\end{enumerate}
\end{R2.1c}

\begin{proof}
Let $\{g_i\,|\,i\in\Lambda\}$ be a net in $G$ with $|g_i|\nearrow\infty$ and $\K(g_i)\rightharpoonup\C$. Let $K$ be a finite subset of $G$ with $e\in K$. We shall find an element $t\in G$ with $Kt\subseteq\C$. Since $G$ is generated by $\Gamma$, there exists an integer $r\ge1$ such that $K\subseteq\Gamma^r$. For each $i\in\Lambda$, let $\ell_i=|g_i|$ ($\nearrow\infty$) and write
$$
g_i=\gamma_1\dotsm\gamma_{\ell_i-r-1}\gamma_{\ell_i-r}\dotsm\gamma_{\ell_i}\in\Gamma^{\ell_i},\quad\alpha_i=\gamma_{\ell_i-r}\dotsm\gamma_{\ell_i}\in\Gamma^{r+1}.
$$
Then $K\alpha_i\subseteq\K(g_i)$ for $i\in\Lambda$ with $\ell_i>r$. Since $\Gamma^{r+1}$ is compact, we can assume (a subnet of) $\alpha_i\to t$ in $G$. Thus $t\in\C$ and $Kt\subseteq\C\not=\emptyset$ by Lemma~\ref{R1.1}. This also proves that $\C$ is a discretely thick subset of $G$. The proof is completed.
\end{proof}

This simple observation will be very useful. From the choice of $\alpha_i$ in its proof above we can easily conclude the following lemma:

\begin{R2.1d}[{cf.~\cite[Lem.~3.6]{R2}}]
Let $F$ be any finite subset of $G$. Then there is an integer $n>1$ such that for all $g\in G$ with $|g|\ge n$, there is an element $t\in G$ such that $|t|=n$ and $Ft\subseteq\K(g)$.
\end{R2.1d}

Note that if $r_i\to\infty$ and $|g_i|\to\infty$ with no any restriction on $|g_i|-r_i$, then it is possible that $\Gamma^{r_i}g_i\rightharpoonup\emptyset$ even though $G=\mathbb{Z}$ or $\mathbb{R}$. Moreover, the compactness of the generator set $\Gamma$ plays an important role in the above proof of Lemma~\ref{R2.1}c and also in Theorem~\ref{R2.6}a below.
\end{se}

\begin{se}[recurrence]\label{R2.2}
Let $(G,X)$ be a flow with a compactly generated phase group $G$ that has a compact generating set $\Gamma$ and let $x\in X$. We will define recurrence at $x$ in two ways.
\begin{R2.2a}
We say that $x$ is \textit{$\Gamma$-recurrent of type I} under $(G,X)$ if $\C x\cap U\not=\emptyset$ for all $\Gamma$-cone $\C$ in $G$ and every $U\in\mathfrak{N}_x$. We say that $x$ is \textit{recurrent of type I} under $(G,X)$ if $x$ is $\Gamma$-recurrent of type I for all compact generating set $\Gamma$ of $G$. Clearly, if $x$ is recurrent of type I, then $tx$ is recurrent of type I for every $t\in G$.
\end{R2.2a}

\begin{R2.2b}
We say that $x$ is \textit{$\Gamma$-recurrent of type II} under $(G,X)$ if for each $U\in\mathfrak{N}_x$ and every net $\{g_i\}$ in $G$ with $|g_i|\to\infty$ there exists an integer $n\ge1$ such that there is a subnet $\{g_{i_j}\}$ from $\{g_i\}$ and $c_j\in\K(g_{i_j})$ with $c_jx\in U$ and $|c_j|\le n$. See \cite[Def.~1.1]{R2}.
\end{R2.2b}
\end{se}

\begin{lem}\label{R2.3}
Let $(G,X)$ be a flow, where $G$ is compactly generated by $\Gamma$, and, $x\in X$. Then:
\begin{enumerate}[$(1)$]
\item If $x$ is an a.p point, then $x$ is $\Gamma$-recurrent of type II (cf.~\cite[Prop.~3.7(iv)]{R2}).
\item If $x$ is $\Gamma$-recurrent of type II, then $x$ is $\Gamma$-recurrent of type I.
\end{enumerate}
\end{lem}

\begin{proof}
(1). Let $U\in\mathfrak{N}_x$ and simply write $A=N_G(x,U)$. Then $A$ is a discretely syndetic subset of $G$ by \ref{R1.5}a. So there is a finite set $F$ in $G$ such that $F^{-1}A=G$, or equivalently, $Ft\cap A\not=\emptyset$ for all $t\in G$. Let $n_1=\max\{|f|+1\,|\,f\in F\}$. Let $\{g_i\}$ be a net in $G$ with $|g_i|\to\infty$.
By Lemma~\ref{R2.1}d, there exists an integer $n>n_1$ such that there is a subnet $\{g_{i_j}\}$ from $\{g_i\}$ such that for each $j$, there is an element $t_j\in G$ with $Ft_j\subseteq\K(g_{i_j})$ and $|t_j|=n$. Take $f_j\in F$ with $c_j=f_jt_j\in A$, then $|c_j|\le2n$ and $c_jx\in U$ for all $j$. Thus $x$ is $\Gamma$-recurrent of type II.

(2). Suppose $x$ is $\Gamma$-recurrent of type II. Let $\C$ be a $\Gamma$-cone in $G$. That is, $\K(g_i)\rightharpoonup\C$, where $\{g_i\}$ is a net in $G$ with $|g_i|\nearrow\infty$. Let $U, V\in\mathfrak{N}_x$ with $\overline{V}\subseteq U$. By Def.~\ref{R2.2}b, there is an integer $n\ge1$ and a subnet $\{g_{i_j}\}$ from $\{g_i\}$ and $c_j\in\K(g_{i_j})$ such that $c_jx\in V$ and $|c_j|\le n$. Since $\Gamma^n$ is compact and $c_j\in\Gamma^n$, we can assume (a subnet of) $c_j\to c$ in $G$ and $c_jx\to cx\in\overline{V}$. Since $c\in\C$ by definition and $\C x\cap U\not=\emptyset$, $x$ is $\Gamma$-recurrent of type I.
The proof is complete.
\end{proof}

\begin{lem}[{\cite[Thm.~4.16, 4.17, 4.24]{GH}}]\label{R2.4}
A flow $(G,X)$ is locally weakly a.p with $G$ under the discrete topology if and only if $\mathscr{O}_G\colon X\rightarrow2^X$ is upper semi-continuous.
\end{lem}

\begin{proof}
Sufficiency. Assume $\mathscr{O}_G$ is upper semi-continuous. By \ref{R1.3}b and Lemma~\ref{R1.4}a, every orbit-closure is minimal so $(G,X)$ is pointwise a.p. Let $\alpha\in\mathscr{U}$ and $x\in X$. Let $U\in\mathfrak{N}_x$ be so small that $\alpha[y]\supseteq U$ for all $y\in U$. Since $\overline{Gx}$ is compact minimal, there exists a finite set $\{z_1,\dotsc,z_k\}\subset Gx$ with $W_i\in\mathfrak{N}_{z_i}$, $t_i\in G$ such that $t_iW_i\subseteq U$ and $\overline{Gx}\subseteq\bigcup_{i=1}^kW_i$. There is a $V\in\mathfrak{N}_x$ with $V\subseteq U$ such that $\overline{Gy}\subseteq\bigcup_{i=1}^kW_i$ for all $y\in V$. Then for $F=\{t_1,\dotsc,t_k\}$, $Fty\cap\alpha[y]\not=\emptyset$ for all $y\in V$ and $t\in G$. Thus $(G,X)$ is locally weakly a.p with $G$ under the discrete topology.

Necessity. Suppose $(G,X)$ is locally weakly a.p with $G$ under the discrete topology and let $x\in X$ and $N$ an open neighborhood of $\overline{Gx}$ such that $\overline{N}$ is compact. It is not difficult to show that $\overline{GN}$ is compact.

Indeed, for each $y\in\overline{N}$ and each compact $U_y\in\mathfrak{N}_y$, there is a finite set $K_y$ in $G$ and an open set $V_y\in\mathfrak{N}_y$ such that $GV_y\subseteq K_y^{-1}U_y$. Since $\overline{N}$ is compact, there is a finite set $\{y_1,\dotsc,y_n\}\subset\overline{N}$ with $\overline{N}\subseteq V_{y_1}\cup\dotsm\cup V_{y_n}$ so that $GN\subseteq\bigcup_{i=1}^nK_{y_i}^{-1}U_{y_i}$. Thus $\overline{GN}$ is compact.

Now replacing $X$ by $\overline{GN}$ if necessary, we may assume $X$ is compact. It is enough to find a $V\in\mathfrak{N}_x$ such that $GV\subseteq N$. For this, choose a closed neighborhood $W$ of $\overline{Gx}$ with $W\subset N$. For $X-W$ is open and $X-N$ is compact and $X-N\subset X-W$, there exists a finite set $K\subset G$ such that $G(X-N)\subseteq K(X-W)$. Take $V\in\mathfrak{N}_x$ so small that $K^{-1}V\subseteq W$. Then $K^{-1}V\cap(X-W)=\emptyset$, $V\cap K(X-W)=\emptyset$, and $V\cap G(X-N)=\emptyset$. So $GV\cap(X-N)=\emptyset$ and $GV\subseteq N$. The proof is complete.
\end{proof}

It should be mentioned that the Lagrange stability plays an important role for the sufficiency of Lemma~\ref{R2.4}. The main ideas of the proof of Theorem~A are contained in Lemma~\ref{R2.4} above and Lemma~\ref{R2.5} below.

\begin{lem}\label{R2.5}
Let $(G,X)$ be any flow with $G$ compactly generated by $\Gamma$, where $\overline{Gy}$ need not be compact for $y\in X$.
Let $x\in X$ and suppose there is an open neighborhood $U$ of $\overline{Gx}$ with $\overline{U}$ compact, such that $x\notin\texttt{int}\,U_G^\infty$.
Then there are $g_i^{}\in G$ with $|g_i^{}|\to\infty$ and $x_i\in X$ with $x_i\to x$, such that $g_i^{}x_i\to y\in X\setminus U$ and $\K(g_i^{-1})\rightharpoonup\C$ with $\C y\subseteq\overline{U}$ and ${\left(\overline{\C y}\right)}_G^\infty\not=\emptyset$.
\end{lem}

\begin{proof}
Since $x\notin\texttt{int}\,U_G^\infty$, there is a net $x_i\in U\setminus U_G^\infty$ with $x_i\to x$. We can obviously take $g_i^{}\in G$ such that
$g_i^{}x_i\notin U$ and $|g_i^{}|=\min\{|t|\colon t\in G, tx_i\notin U\}$.
Obviously (a subnet of) $|g_i^{}|\to\infty$. For all $i$ we write
$g_i^{}=\gamma_{i,1}\gamma_{i,2}\dotsm\gamma_{i,|g_i^{}|}$ and $g_i^\prime=\gamma_{i,2}\dotsm\gamma_{i,|g_i^{}|}$.
Then $g_i^\prime x_i\in U$ and $\gamma_{i,1}\in\Gamma$ for all $i$. Since $\overline{U}$ and $\Gamma$ are compact, we can assume (a subnet of)
$\gamma_{i,1}\to\gamma$ and $g_i^\prime x_i\to y^\prime\in U$. Let $y=\gamma y^\prime$; then $g_i^{} x_i\to y\notin U$.

Let $\K(g_i^{-1})\rightharpoonup\C$, a $\Gamma$-cone in $G$. Given $c\in \C$, we can take a subnet $\{g_{i_j}^{}\}$ from $\{g_i^{}\}$ and $c_j=k_jg_{i_j}^{-1}$ with $k_j\in\Gamma^{|g_{i_j}^{}|-1}$ and $c_j\to c$. Thus by $|k_j|<|g_{i_j}^{}|$, it follows that
$$
cy={\lim}_jc_jy={\lim}_j (k_jg_{i_j}^{-1})(g_{i_j}^{}x_{i_j})\in \overline{U}.
$$
So $\overline{\C y}\subseteq \overline{U}$. Finally, to show $\overline{\C y}$ contains a $G$-orbit, let $\mathcal {F}$ be the collection of finite subsets of $G$ with a partial order by inclusion ($F\le F^\prime$ if $F\subseteq F^\prime$). Then $\mathcal {F}$ is a directed system. Since $\C$ is discretely thick in $G$ by Lemma~\ref{R2.1}c, we have for $F\in\mathcal {F}$ that there exists an element $t_F\in\C$ with $Ft_F\subseteq\C$. Since $\overline{U}$ is compact and $t_Fy\in \overline{U}$, we can assume (a subnet of) $t_Fy\to y^\prime\in\overline{U}$. Then for any fixed $t\in G$ there is some $F_0\in\mathcal {F}$ with $t\in F_0$. Then $t\in F$ for all $F\ge F_0$ so that
$t\in F$ eventually and $tt_Fy\in\C y$ eventually. Hence $ty^\prime\in\overline{\C y}$ for all $t\in G$.
Then $\overline{Gy^\prime}\subseteq\overline{\C y}$.
The proof is completed.
\end{proof}

\begin{se}\label{R2.6}
We are ready to give concise and self-contained proofs of our main theorems.
Recall that $\mathscr{O}_G\colon X\rightarrow 2^X$ is defined by $x\mapsto\overline{Gx}$, and, $U_G^*=\{x\in X\,|\,\overline{Gx}\subseteq U\}$ is closed invariant for all subset $U$ of $X$.

\begin{R2.6a}
Let $(G,X)$ be any flow, where $G$ is compactly generated by $\Gamma$ and $\dim X=0$. Then the following conditions are pairwise equivalent:
\begin{enumerate}[(1)]
\item $(G,X)$ is pointwise $\Gamma$-recurrent of type I.
\item $(G,X)$ is pointwise $\Gamma$-recurrent of type II.
\item $(G,X)$ is pointwise a.p.
\item $X$ is a union of $G$-minimal sets.
\item $R_o(X)$ is closed.
\item $\mathscr{O}_G$ is $\ulcorner\!\textrm{continuous}\!\urcorner$ on $X$.
\item $\mathscr{O}_G$ is upper semi-continuous on $X$.
\item Given any compact-open subset $U$ of $X$, $U_G^*$ is compact open.
\item $(G,X)$ is locally weakly a.p with $G$ under the discrete topology.
\end{enumerate}
\end{R2.6a}

\begin{notes*}
\begin{enumerate}[1.]
\item The above (3) and (5) implies that the quotient mapping $\rho\colon X\rightarrow X/G$ is open and closed by Lemma~\ref{R1.4}b.
\item If in addition $(G,X)$ is topologically transitive (i.e. $\overline{GU}=X$ for every non-empty open set $U$ in $X$), then each of conditions (1)\,$\thicksim$\,(9) implies that $(G,X)$ is minimal. For (3) implies that $X/G$ is a singleton.
\end{enumerate}
\end{notes*}

\begin{proof}
(2)$\Rightarrow$(1): By (2) of Lemma~\ref{R2.3}.

(3)$\Rightarrow$(2): By (1) of Lemma~\ref{R2.3}.

(4)$\Rightarrow$(3): By \ref{R1.5}a.

(5)$\Rightarrow$(4): By Lemma~\ref{R1.4}a.

(6)$\Rightarrow$(5): By \ref{R1.3}b.

(7)$\Rightarrow$(6): By Def.~\ref{R1.2}.

(8)$\Rightarrow$(7): Assume (8). Let $x\in X$ and $V$ an open neighborhood of $\overline{Gx}$. Because $X$ is locally compact 0-dimensional and $\overline{Gx}$ is compact, there is a compact-open set $U$ with $\overline{Gx}\subseteq U\subseteq V$. Then $x\in\overline{Gx}\subseteq U_G^*\subseteq V$. Since $U_G^*$ is open, $\mathscr{O}_G$ is upper semi-continuous. Thus (8)$\Rightarrow$(7).

(1)$\Rightarrow$(8): Assume (1). Let $U$ be a compact-open subset of $X$. Then $U_G^*=U_G^\infty$ is a $G$-invariant compact set in $X$. To prove that $U_G^*$ is open, suppose the contrary that $U_G^*$ is not open in $X$. Then $U_G^*\not=\emptyset$ and there is a point $x\in U_G^\infty\setminus\texttt{int}\,U_G^\infty$. Then by Lemma~\ref{R2.5}, it follows that there exists a point $y\in X\setminus U$ and a $\Gamma$-cone $\C$ such that $\C y\cap(X\setminus U)=\emptyset$, contrary to that
 $X\setminus U\in\mathfrak{N}_y$ and that $y$ is $\Gamma$-recurrent of type I by condition (1). Thus (1) implies (8).

(7)$\Leftrightarrow$(9): By Lemma~\ref{R2.4}.

The proof of Theorem~\ref{R2.6}a is thus completed.
\end{proof}

The idea of (1)$\Rightarrow$(8) of using the $\Gamma$-lengths of elements in $G$ could date back to \cite{SS}. It is also useful for the following theorem, where $V_G^\infty=\bigcap_{g\in G}gV$ as in \ref{R1.2}c and in particular there is no condition $\dim X=0$.

\begin{R2.6b}
Let $(G,X)$ be any flow with $G$ compactly generated by $\Gamma$, where $\overline{Gx}$ need not be compact for $x\in X$.
Let $V$ be an open subset of $X$ such that $V_G^\infty$ is compact. Then
$V_G^\infty$ is open if and only if $\overline{Gx}\cap V_G^\infty=\emptyset$ for every $x\in X\setminus V_G^\infty$.
\end{R2.6b}

\begin{proof}
Necessity is obvious, for $X\setminus V_G^\infty$ is a closed $G$-invariant subset of $X$. To show sufficiency, suppose the contrary that $V_G^\infty$ is not open (so $V_G^\infty\not=\emptyset$). Since $X$ is a locally compact Hausdorff space and $V_G^\infty$ is compact, we can find an open neighborhood $U$ of $V_G^\infty$ with $\overline{U}\subseteq V$ such that $\overline{U}$ is compact. Clearly, $V_G^\infty=\bigcap_{g\in G}gU=U_G^\infty$. Then $U_G^\infty\setminus\texttt{int}\,U_G^\infty\not=\emptyset$. Further by Lemma~\ref{R2.5}, it follows that there exists some point $y\in X\setminus V_G^\infty$ and some $\Gamma$-cone $\C$ in $G$ such that
$\C y\subseteq\overline{U}\subseteq V$ so that $\overline{Gy}\cap V_G^\infty\not=\emptyset$, contrary to the sufficiency condition.
The proof is completed.
\end{proof}

\begin{R2.6c}
Let $(G,X)$ be a flow with $G$ compactly generated. Then $\mathscr{O}_G$ is $\ulcorner\!\textrm{continuous}\!\urcorner$ if and only if $\mathscr{O}_G$ is upper semi-continuous.
\end{R2.6c}

\begin{proof}
Sufficiency is obvious by Def.~\ref{R1.2}a and \ref{R1.2}b. Necessity follows easily from Lemma~\ref{R2.5}. We omit the details here.
\end{proof}
\end{se}

\begin{thm}[{cf.~\cite[Thm.~2.32, Thm.~4.29]{GH} for $G$ a topological group}]\label{R2.7}
Let $(G,X)$ be a minimal flow and $H$ a syndetic normal subgroup of $G$. Then $\mathscr{O}_H\colon X\rightarrow 2^X$ is continuous; that is, $(H,X)$ is a weakly a.p flow with $H$ under the discrete topology.
\end{thm}

\begin{proof}\footnote{Note here that $G$ is only a para-topological group. So if $K$ is a compact set in $G$, $K^{-1}$ need not be compact. If $G$ is a topological group, the proof of this theorem might be simplified, and moreover, the locally weakly a.p is independent of the topology of $G$.}
Let $K$ be a compact subset of $G$ with $G=K^{-1}H$ (cf.~\ref{R1.5}). Let $x,y\in X$. By $\overline{Gx}=X$ there is a net $t_n\in G$ such that $y=\lim_nt_nx$. Take $s_n\in K$ with $s_n\to s\in K$ such that $s_nt_n\in H$. Since $G\times X\rightarrow X$ is continuous, it follows that $sy=\lim s_nt_nx\in\overline{Hx}$. Hence $y\in s^{-1}\overline{Hx}=\overline{Hs^{-1}x}$. So by Lemma~\ref{R1.5}e, $\{k^{-1}\overline{Hx}\,|\,k\in K\}$ is a partition of $X$ into $H$-minimal sets. Let $k_n^{-1}x\to k^{-1}x$ with $k_n,k\in K$ and let $k_n^{-1}\overline{Hx}\to L$. We need prove $L=k^{-1}\overline{Hx}$. We may assume $k_n\to\ell\in K$. Then $x=\ell k^{-1}x$ (equivalently $\ell^{-1}x=k^{-1}x$), and, $\overline{Hx}=\ell L$ for $G\times2^X\rightarrow 2^X$ is continuous. Thus, $L=\ell^{-1}\overline{Hx}=\overline{H\ell^{-1}x}=\overline{Hk^{-1}x}=k^{-1}\overline{Hx}$. Then $\mathscr{O}_H$ is continuous. The proof is complete.
\end{proof}

\begin{rems}\label{R2.8}
\begin{enumerate}[(1)]
\item It should be noticed that since we do not know whether the $\Gamma$-recurrence of type I at a point of $X$ is equivalent to that of type II for non-discrete $G$, so $(1)\Leftrightarrow(2)$ in Theorem~\ref{R2.6}a is not trivial. In addition, Theorem~\ref{R2.6}b can be utilized for proving $(1)\Rightarrow(8)$ in Theorem~\ref{R2.6}a.
\item If $X$ has a neighborhood base at a point $x_0$ of clopen sets instead of $\dim X=0$, is $x_0$ a.p whenever $(G,X)$ is pointwise $\Gamma$-recurrent of type I?
\end{enumerate}
\end{rems}
\section{Distality, equicontinuity and regular almost periodicity}\label{secR.3}
The main goal of $\S$\ref{secR3.1} is to prove Theorem~B stated in $\S$\ref{sec0} using Theorem~\ref{R2.6}a.
According to Furstenberg's structure theorem~\cite{F} every minimal distal compact flow has an equicontinuous factor. In $\S$\ref{secR3.1} we also consider the problem:
{\it When does a distal compact flow have an equicontinuous or regularly a.p factor?}
Note here that without minimality of $X$ the classical F-topology and $\tau$-topology arguments (cf.~\cite{F, V77} and \cite{EGS, G76, CD}) are invalid.
Moreover, we consider weakly rigid $\mathbb{Z}$-flows in $\S$\ref{secR3.2} and pointwise periodic $\mathbb{Z}$-flows in $\S$\ref{secR3.3}.
\subsection{Distal 0-dimensional flows}\label{secR3.1}
It is easy to prove that if $(G,X)$ is equicontinuous, then $(G,X)$ is distal. No converse is valid in general. However, using Theorem~\ref{R2.6}a we can then conclude the following generalization of an important theorem of R.~Ellis:

\begin{sthm}\label{R3.1.1}
Let $(G,X)$ be a flow such that $\dim X=0$ and  $G$ is compactly generated by $\Gamma$.
Then the following conditions are pairwise equivalent:
\begin{enumerate}[(1)]
\item $(G,X)$ is pointwise product $\Gamma$-recurrent of type I (i.e. $(G,X\times X)$ is pointwise $\Gamma$-recurrent of type I).
\item $(G,X)$ is a distal flow.
\item $(G,X)$ is an equicontinuous flow.
\end{enumerate}
\end{sthm}

\begin{notes*}
\begin{enumerate}[1.]
\item As to (2)$\Leftrightarrow$(3), see~\cite[Theorem~2]{E58} for $G$ generative, \cite[Corollary~(3.11)]{MW76} for $(G,X)$ a minimal flow with $G$ a direct product of a compactly generated separable group with a compact group, and \cite[Corollary~1.9]{AGW} for $G$ finitely generated.

\item Let $T$ be a Hausdorff topological group compactly generated by $\Upsilon$ and $\mathbb{K}$ a compact Hausdorff topological group. Let $G=T\times \mathbb{K}$. Then $G=\left(\bigcup_{r=1}^\infty\Upsilon^r\right)\times\mathbb{K}=\bigcup_{r=1}^\infty(\Upsilon\times\mathbb{K})^r$. Since $\Gamma=\Upsilon\times \mathbb{K}$ is compact in $G$, $G$ is compactly generated by $\Gamma$. So Theorem~\ref{R3.1.1} generalizes the result of McMahon-Wu~\cite[Corollary~(3.11)]{MW76}.

\item Equicontinuity implies the pointwise regional distality under our situation. However, since $X$ need not be compact here, it need not imply equicontinuity that $(G,X)$ is pointwise regionally distal ($\textrm{RP}[x]=\{x\}\ \forall x\in X$) in general. In view of this, the approach of McMahon-Wu \cite{MW76} (using the Furstenberg tower of minimal distal flow) does not work for the non-minimal case.

\item To what extent can the condition of zero-dimensionality be relaxed for ``$(1)\Rightarrow(2)$'' in Theorem~\ref{R3.1.1}? By a slight modification of the proof of \cite[Proposition~6.5]{GM}, we can show that there exists a homeomorphism $f\colon X\rightarrow X$ on a compact metric space $X$ with $\dim X\ge1$ such that $(f,X)$ is minimal, weakly mixing, and $f^{n_k}\to \textrm{id}_X$ and $f^{m_k}\to\textrm{id}_X$ uniformly w.r.t. some sequences $n_k\nearrow+\infty$ and $m_k\searrow-\infty$. Thus $(f,X)$, as a $\mathbb{Z}$-flow, is pointwise product $\{\pm 1,0\}$-recurrent of type I (see Proposition~\ref{R4.6}) but it is not distal (for it is weakly mixing).
\end{enumerate}
\end{notes*}

\begin{proof}
``$(3)\Rightarrow(2)\Rightarrow(1)$'' is obvious. The remainder is to prove that ``$(1)\Rightarrow(3)$''. Now assume (1); then $(G,X\times X)$ is pointwise $\Gamma$-recurrent of type I with $\dim X\times X=0$. By Theorem~\ref{R2.6}a, the orbit-closure relation $R_o$ of $(G,X\times X)$ is closed and so by Lemma~\ref{R1.4}c, it follows that $(G,X)$ is equicontinuous.
The proof is completed.
\end{proof}

\begin{sthm}\label{R3.1.2}
Let $(G,X)$ be a compact flow, where $X$ is not connected and $G$ is compactly generated. If $(G,X)$ is distal, then $(G,X)$ has an equicontinuous non-trivial factor $(G,Y)$.
\end{sthm}

\begin{proof}
Define the ``component relation'' $R_c$ of $X$ as follows: $(x,x^\prime)\in R_c$ if and only if $x$ and $x^\prime$ are in the same connected component of $X$. Then it is not difficult to verify that $R_c$ is an $G$-invariant closed equivalence relation on $X$ (cf.~\cite[Def.~2.2 and Proposition~2.3]{MW76}). Set $Y=X/R_c$ equipped with the quotient topology and let $\rho\colon X\rightarrow Y$ be the canonical quotient mapping. Then $(G,Y)$ is a distal flow with $Y$ to be compact 0-dimensional, and, $\rho\colon(G,X)\rightarrow(G,Y)$ is an extension. Since $X$ is not connected, so $Y$ is not a singleton. By Theorem~\ref{R3.1.1}, $(G,Y)$ is equicontinuous. The proof is completed.
\end{proof}

\begin{sthm}[{cf.~\cite[Thm.~3]{E58} for $G$ generative; \cite[Problem~(1)]{E58}}]\label{R3.1.3}
Let $(G,X)$ be a distal flow with $G$ compactly generated, such that $\dim\overline{Gx}=0$ for all $x\in X$, where $X$ need not be locally compact. Then, under the discrete topology of $G$, $(G,X)$ is pointwise regularly a.p, and, $(G,\overline{Gx})$ is a regularly a.p subflow for every $x\in X$.
\end{sthm}

\begin{note*}
A normal discretely syndetic closed subgroup of $G$ is clopen and has the finite index in $G$. However, the subgroup itself need not be discrete.
\end{note*}

\begin{proof}
Let $x\in X$. Using $\overline{Gx}$ instead of $X$, we can assume $\overline{Gx}=X$ with $\dim X=0$. By Theorem~\ref{R3.1.1}, $(G,X)$ is an equicontinuous minimal compact 0-dimensional flow. Let $E(X)$, the Ellis enveloping semigroup of $(G,X)$, be the closure of $G$ in $X^X$ with the topology of pointwise/uniform convergence. Then $E(X)$ is a compact topological group~\cite[4.4]{E69}. Since $\dim X=0$, we can take $U\in\mathfrak{N}_{x}$ to be clopen. Since $E(X)\rightarrow X$, defined by $p\mapsto px$ with $e=\textrm{id}_X\mapsto x$, is continuous, thus $\mathbb{U}=\{p\,|\,p\in E(X), px\in U\}$ is a clopen neighborhood of $e$ in $E(X)$. By \cite[Theorem~7.6]{HR},\footnote{\textbf{Theorem}\,(Hewitt and Ross; cf.~\cite[Thm.~7.6]{HR}).\quad {\it Let $G$ be a compact topological group and let $\mathds{U}$ be a clopen neighborhood of the identity. Then $\mathds{U}$ contains a clopen normal subgroup $N$ of $G$ and $G/N$ is finite.}\label{f4}} it follows that there exists a clopen normal subgroup $N$ of $E(X)$ with $N\subseteq \mathbb{U}$ and $E(X)/N$ is finite. Since $G$ is dense in $E(X)$, there is a finite subset $K$ of $G$ such that $KN=E(X)$. Set $H=N\cap G$. Then $KH=G$ and $H$ is a closed normal discretely syndetic subgroup of $G$. By $H\leqq N\subseteq \mathbb{U}$, $Hx\subseteq U$ and $x$ is a regularly a.p point under $(G,X)$ with discrete phase group.
\end{proof}

\begin{cor*}
Let $(G,X)$ be a minimal distal flow with $G$ compactly generated. If $X$ is not connected, then $(G,X)$ is has a regularly a.p factor.
\end{cor*}

\begin{sthm}\label{R3.1.4}
Let $(G,X)$ be a metric flow, which is pointwise regularly a.p under the discrete topology of $G$. Suppose $G$ contains only countably many clopen normal subgroups of finite index. Then the set of points $R$ at which $(G,X)$ is equicontinuous is residual in $X$.
\end{sthm}

\begin{proof}
Let $\{H_n\,|\,n=1,2,\dotsc\}$ be the set of clopen normal subgroups with finite index of $G$ and let $\rho$ be a metric on $X$. For all $m,n\in\mathbb{N}$ set
$E(n,m)=\{x\in X\,|\,H_nx\subseteq B(x,1/m)\}$, where $B(x,r)=\{y\,|\,\rho(x,y)\le r\}$ for $r>0$.
Then $E(n,m)$ is a closed subset of $X$ and $\bigcup\{E(n,m)\,|\,n=1,2,\dotsc\}=X$. Hence $E(m):=\bigcup_{n}\textrm{int}\,E(n,m)$ is an everywhere dense open subset of $X$. Let $E=\bigcap_mE(m)$. Then $E$ is a residual subset of $X$. We will prove that $E\subseteq R$.

First, from the definition of $E$, it follows that given any $U\in\mathfrak{N}_x$, $x\in E$, there exists a $V\in\mathfrak{N}_x$ and a discretely syndetic normal subgroup $A$ of $G$ such that $AV\subseteq U$. 

Let $x\in E$. Let $U\in\mathfrak{N}_x$ be compact; then there are $V\in\mathfrak{N}_x$ and $A$ a discretely syndetic subset of $G$ such that $AV\subseteq U$. Let $K$ be a finite subset of $G$ such that $K^{-1}A=G$.
Since $K$ is finite, it follows that $K^{-1}(U\times U)$ is compact in $X\times X$.
Then
$$G(V\times V)=K^{-1}A(V\times V)\subseteq K^{-1}(U\times U)\subseteq G(U\times U)$$
shows that $\overline{G(V\times V)}$ is compact with $\overline{G(V\times V)}\subseteq G(U\times U)$ and that
$$
\bigcap\{\overline{G(V\times V)}\,|\, V\in\mathfrak{N}_x\textrm{ and }V\subseteq U\}=\bigcap\{G(V\times V)\,|\, V\in\mathfrak{N}_x\textrm{ and }V\subseteq U\}.
$$
Then the proof that $(G,X)$ is equicontinuous at $x$ may be completed as in Proof of \cite[Lemma~1]{G56}.

Indeed, suppose the contrary that $G$ is not equicontinuous at $x$; then there exists an open index $\varepsilon\in\mathscr{U}$ such that $G(N\times N)\nsubseteq \varepsilon$ for every $N\in\mathfrak{N}_x$. We set $\varepsilon^\prime=X\times X-\varepsilon$ and define
$\mathfrak{F}=\{G(N\times N)\cap \varepsilon^\prime\,|\,N\in\mathfrak{N}_x\}$. Since $\bar{\mathfrak{F}}=\{\overline{F}\,|\,F\in\mathfrak{F}\}$ has the finite intersection property, hence $\emptyset\not=\bigcap\bar{\mathfrak{F}}=\bigcap\mathfrak{F}$. This implies that $G$ is not distal on $X$. However, as $(G,X)$ is pointwise regularly a.p under the discrete topology of $G$, it follows by Lemma~\ref{R1.5}d that $(G,X)$ is distal. The proof is complete.
\end{proof}

By a slight modification of the above proof with `$K$ compact' instead of `$K$ finite' and using Lemma~\ref{R1.9}, we can present another formulation of Theorem~\ref{R3.1.4} as follows:

\begin{R3.1.4'}
Let $(G,X)$ be a distal pointwise regularly a.p flow with $X$ a metric space. Suppose $G$ contains only countably many normal subgroups. Then the set of points $R$ at which $(G,X)$ is equicontinuous is residual in $X$.
\end{R3.1.4'}

\begin{cor*}[{cf.~\cite[Thm.~4]{E58} for $G$ generative; \cite[Problem~(1)]{E58}}]
Let $(G,X)$ be a distal metric flow
such that $\dim \overline{Gx}=0$ for all $x\in X$. Suppose $G$ is compactly generated and contains only countably many normal clopen subgroups of finite index. Then the set of points $R$ at which $(G,X)$ is equicontinuous is a residual subset of $X$.
\end{cor*}

\begin{proof}
This follows easily from  Theorem~\ref{R3.1.3} and Theorem~\ref{R3.1.4}.
\end{proof}

\begin{exa}[{An explicit realization of \cite[Exa.~(3.4.1)]{MW76}}]\label{R3.1.5}
Let $\mathbb{Z}_2=\{0,1\}$ be the cyclic group of order $2$ ($1+1=0$). Let $Y=\prod_{-\infty}^{+\infty}\mathbb{Z}_2$. Then $Y$ is a compact totally disconnected topological group under the product topology. Let $T=\{t\in Y\,|\,t(n)=0\textrm{ for all but a finite set of $n$'s}\}$. Clearly, $T$ is a dense subgroup of $Y$. Now we can define an equicontinuous minimal flow $(T,Y)$ by
$$
T\times Y\rightarrow Y,\quad (t,y)\mapsto ty=(t(n)+y(n))_{n\in\mathbb{Z}}.
$$
Let $X=Y_{+1}\cup Y_{-1}$, where $Y_{+1}$ and $Y_{-1}$ are two copies of $Y$. Let $o\in Y$ with $o(n)=0$ for all $n\in\mathbb{Z}$. For $i\ge1$ take $\xi_i,\eta_i\in Y$ such that
$$
\xi_i(n)=\begin{cases}0& \textrm{if }n\le i,\\1 & \textrm{if }n>i;\end{cases}\quad \eta_i(n)=\begin{cases}1& \textrm{if }n<-i,\\0 & \textrm{if }n\ge -i.\end{cases}
$$
Then $\xi_i\to o$ and $\eta_i\to o$ as $i\to\infty$. Moreover, the cylinders
$$
[0_{-j},\dotsc,0_j,1_{j+1}]=\{y\in Y\,|\,y(n)=0\textrm{ for }|n|\le j\textrm{ and }y(j+1)=1\},\quad j=1,2,\dotsc,
$$
are disjoint clopen subsets of $Y$ such that
$\xi_i\in[0_{-i},\dotsc,0_i,1_{i+1}]$ and $\eta_i\notin\bigcup_{j=1}^\infty[0_{-j},\dotsc,0_j,1_{j+1}]$ for all $i\ge1$.
For each $i\ge1$, define $b_i\colon X\rightarrow X$ by: for $y_\epsilon\in Y_\epsilon$, $\epsilon=\pm1$, put
$$
b_iy_\epsilon=\begin{cases}y_\epsilon & \textrm{if }y_\epsilon\notin[0_{-i},\dotsc,0_{i}, 1_{i+1}]_\epsilon,\\y_{-\epsilon} & \textrm{if }y_\epsilon\in[0_{-i},\dotsc,0_{i}, 1_{i+1}]_\epsilon;\end{cases}
$$
and for $t\in T$, define $t\colon X\rightarrow X$ by $ty_\epsilon=(ty)_\epsilon$. Let $G$ be the countable discrete group generated by $T$ and $\{b_i\}_{i=1}^\infty$. Then $(G,X)\xrightarrow{\pi\colon y_\epsilon\mapsto y}(G,Y)$ is a 2-to-1 distal/equicontinuous extension of the minimal equicontinuous $(G,Y)$; moreover, $(o_{+1},o_{-1})$ is a regionally proximal pair under $(G,X)$, hence $(G,X)$ is a minimal distal non-equicontinuous (non-a.a) flow, where $X$ is a 0-dimensional compact metric space and $G$ is not compactly generated. In fact, the regionally proximal cell $\textrm{RP}[y_\epsilon]=\{y_\epsilon,y_{-\epsilon}\}$ for all $y_\epsilon\in X$ (noting $\textrm{RP}_\pi=\Delta_X$). Moreover, although $X$ is a homogeneous space, yet $(G,X)$ is not dynamically homogeneous (i.e. $\textrm{Aut}\,(G,X)x\not=X$ for $x\in X$; otherwise, $(G,X)$ is equicontinuous by Auslander \cite[Theorem~2.13]{A88}).
\end{exa}

\begin{exa}[{a modification of D.~McMahon's example; cf.~\cite[Exa.~(3.4.2)]{MW76}}]\label{R3.1.6}
As in Example~\ref{R3.1.5}, let $Y=\prod_{-\infty}^{+\infty}\mathbb{Z}_2$ be the direct product such that $y_1y_2=(y_1(n)+y_2(n))_{n\in\mathbb{Z}}$ for all $y_1,y_2\in Y$. Let $o\in Y$ such that $o(n)\equiv0$ for all $n\in\mathbb{Z}$, and, set $0^\prime=1$ and $1^\prime=0$. Let $X=Y\times\mathbb{Z}_2$ and $\delta=0$ or $1$. Then $X$ and $Y$ both are totally disconnected compact metric spaces such that $\pi\colon(y,\delta)\mapsto y$ from $X$ onto $Y$ is open 2-to-1 continuous.
Given $i\in\mathbb{Z}$ we define the dual homeomorphisms on $Y$ and $X$, respectively, as follows:
$$
\theta_i\colon Y\rightarrow Y,\quad y\mapsto y_i^\prime,\quad \textrm{where }y_i^\prime(n)=y(n)\textrm{ if }n\not=i\textrm{ and }={y(n)}^\prime \textrm{ if }n=i;
$$
and
$$
\theta_i\colon X\rightarrow X,\quad (y,\delta)\mapsto(y_i^\prime, \delta+y(i)).\footnote{If $\theta_i\colon(y,\delta)\mapsto(y_i^\prime, \delta+y(i-1)+y(i+1))$ for each $i\in\mathbb{Z}$, then we return to McMahon's case. Here our definition simplifies the proof of $\theta_i\circ\theta_j=\theta_j\circ\theta_i$.}
$$
Clearly, for all $i,j\in\mathbb{Z}$, it holds that $\theta_j\circ\theta_i=\theta_i\circ\theta_j$, $\theta_i^2|Y=\textrm{id}$ and $\theta_i^4|X=\textrm{id}$.

Let $G=\bigoplus_{i\in\mathbb{Z}}H_i$, where $H_i=\mathbb{Z}_4$.
Now we may define actions of $G$ on $Y$ and $X$ as follows: for all
$$
t=(t(i))_{i\in\mathbb{Z}}=(\dotsc,0,0,t(i_1),\dotsc,t(i_n),0,0,\dotsc)\in G,
$$
set
$$
t|Y=\theta_{i_1}^{t(i_1)}\circ\dotsm\circ\theta_{i_n}^{t(i_n)}\colon Y\rightarrow Y\quad \textrm{and}\quad t|X=\theta_{i_1}^{t(i_1)}\circ\dotsm\circ\theta_{i_n}^{t(i_n)}\colon X\rightarrow X.
$$
Then $(G,Y)$ is a minimal equicontinuous flow and $\pi\colon (G,X)\rightarrow(G,Y)$ is a 2-to-1 equicontinuous extension. Next we shall prove that $(G,X)$ is a minimal distal non-equicontinuous flow. Since $\pi$ is distal and $Y$ is minimal distal, $(G,X)$ is distal and pointwise a.p having at most two minimal subsets. By $\theta_i^2(o,0)=(o,1)$, it follows that $(G,X)$ is minimal distal. For each integer $j\ge1$ take a point $y_j\in Y$ such that
$$
y_j(n)=0\textrm{ for }|n|\le j\quad \textrm{and}\quad y_j(n)=1\textrm{ for }|n|>j.
$$
Then $(y_j^\prime,1)\to(o,1)$, $(y_j,0)\to(o,0)$ in $X$ and $t_j((y_j^\prime,1),(y_j,0))\to((o,0),(o,0))$ in $X\times X$ as $j\to\infty$, where $t_j=(\dotsc,0,0,1,0,0,\dotsc)\in G$ with $1$ is at the $j$-coordinate. Hence $(o,1)$ is regionally proximal to $(o,0)$ under $(G,X)$. This shows that $(G,X)$ is not equicontinuous (so not locally a.p and not almost automorphic).
\end{exa}

It is well known that distality alone does not imply equicontinuity. Therefore, for equicontinuity, none of the conditions\,---\,$\dim X=0$, $G$ compactly generated and $(G,X)$ distal, of Theorem~\ref{R3.1.1} is surplus.

\begin{srem}
$(G,X)$ in Example~\ref{R3.1.6} has no regularly a.p points. In addition, $(G,Y)$ in Example~\ref{R3.1.6} is regularly a.p (see Proof of Theorem~\ref{R3.1.3}). So the regularly a.p cannot be lifted by finite-to-one covering maps if no restriction to the phase group $G$.
\end{srem}

\subsection{Weakly rigid compact $\mathbb{Z}$-flows}\label{secR3.2}
Let $X$ be a compact Hausdorff space, not necessarily metric, and $f\colon X\rightarrow X$ a self homeomorphism of $X$ in this subsection.
\begin{sse}\label{R3.2.1}
Following \cite{GM} we say that
\begin{enumerate}[(i)]
\item $(f,X)$ is \textit{weakly rigid} if for every $\varepsilon\in\mathscr{U}$ and points $x_1, \dotsc, x_n\in X$ there exists $n\in\mathbb{Z}$, $n\not=0$, such that $(f^nx_i,x_i)\in \varepsilon$ for $i=1,\dotsc,n$.
\item $(f,X)$ is \textit{rigid} (w.r.t. a net $n_k\in\mathbb{Z}$ with $n_k\nearrow\infty$) if $f^{n_k}x\to x$ for all $x\in X$.
\item $(f,X)$ is \textit{uniformly rigid} (w.r.t. a net $n_k\nearrow\infty$) if $f^{n_k}x\to x$ uniformly for all $x\in X$.
\end{enumerate}
Clearly, (iii) $\Rightarrow$ (ii) $\Rightarrow$ (i) but (i) $\not\Rightarrow$ (ii) $\not\Rightarrow$ (iii) in general (cf.~\cite{GM}).
\end{sse}

\begin{sthm}[{cf.~\cite[Prop.~6.7]{GM} for $(2)\Leftrightarrow(5)$ and \cite[Cor.~8.4]{MR} for $(2)\Leftrightarrow(7)$, in the case $(f,X)$ minimal metric}]\label{R3.2.2}
Let $(f,X)$ be such that $\dim X=0$. Then the following are pairwise equivalent:
\begin{enumerate}[(1)]
\item $(f,X)$ is pointwise product $\Gamma$-recurrent of type I, where $\Gamma=\{0,\pm 1\}$.
\item $(f,X)$ is weakly rigid.
\item $(f,X)$ is rigid.
\item $(f,X)$ is uniformly rigid.
\item $(f,X)$ is equicontinuous.
\item $(f,X)$ is positively/negatively pointwise product recurrent (i.e. for every $(x_1,x_2)\in X\times X$ there is a net $n_k\to+\infty$/$n_k\to-\infty$ such that $(f^{n_k}x_1,f^{n_k}x_2)\to(x_1,x_2)$).
\item $(f,X)$ is regularly a.p.
\end{enumerate}
\end{sthm}

\begin{proof}
$(1)\Rightarrow(6)$: By the fact that the only $\Gamma$-cones are $\mathbb{N}$ and $-\mathbb{N}$ in $(\mathbb{Z},+)$.

$(6)\Rightarrow(5)$: Assume (6). Let $V$ be a clopen non-empty subset of $X$. Let $\pi_V\colon X\rightarrow\{0,1\}^\mathbb{Z}$ be defined by $x\mapsto(1_V f^nx)_{n\in\mathbb{Z}}$. Since $V$ is clopen, $\pi_V$ is continuous. Let $Y_V=\pi_VX$ and $\sigma_V\colon Y_V\rightarrow Y_V$ the canonical shift map. Then $\pi_V\circ f=\sigma_V\circ\pi_V$. So $(\sigma_V,Y_V)$ is positively/negatively pointwise product recurrent. This implies that $(\sigma_V,Y_V)$ is distal. For if otherwise, then by \cite[Proposition~5.10]{AD} there were a pair $(y_1,y_2)\in Y_V\times Y_V$ such that $y_1$ is positively/negatively proximal to $y_2$ and $y_1\not=y_2$; further, we can find a positively/negatively asymptotic pair in $Y$ (or by~\cite[Lemma~6.6]{GM}). Thus by Theorem~\ref{R3.1.1}, $(\sigma_V,Y_V)$ is equicontinuous. Let $\mathscr{V}$ be the collection of clopen subsets of $X$. Then $(\sigma,\prod_{V\in\mathscr{V}}Y_V)$ is equicontinuous. Define $\pi\colon X\rightarrow\prod_{V\in\mathscr{V}}Y_V$ by $x\mapsto(\pi_Vx)_{V\in\mathscr{V}}$. Since $X$ is 0-dimensional, $\pi$ is obviously continuous 1-1 with $\pi\circ f=\sigma\circ\pi$. As $(\sigma,\pi X)$ is equicontinuous, it follows that $(f,X)$ is equicontinuous.

$(5)\Rightarrow(1)$: By Theorem~\ref{R3.1.1}.

$(5)\Rightarrow(4)\Rightarrow(3)\Rightarrow(2)$: By definitions.

$(2)\Rightarrow(5)$: By a slight modification of the above proof of $(6)\Rightarrow(5)$.

$(5)\Leftrightarrow(7)$: Assume (5). By Theorem~\ref{R3.1.3}, $(f,X)$ is pointwise regularly a.p. Further by equicontinuity, $(f,X)$ is regularly a.p. Finally, (7)$\Rightarrow$(5) is obvious.

The proof is thus completed.
\end{proof}

It should be noticed that in the above proof of $(6)\Rightarrow(5)$, $\{(\sigma,Y_V)\,|\,V\in\mathscr{V}\}$ need not be a directed system, since we did not define the connecting homomorphism $\phi_V^U\colon Y_U\rightarrow Y_V$ for $U,V\in\mathscr{V}$ with $U\subset V$ or $V\subset U$.

\subsection{Pointwise periodic homeomorphisms}\label{secR3.3}
Let $f\colon X\rightarrow X$ be a homeomorphism. If each point of $X$ is periodic under $f$, then $(f,X)$ is said to be pointwise periodic. If there is some positive integer $\kappa$ such that $f^\kappa=\textrm{id}_X$, then $(f,X)$ is said to be periodic. The least integer $p$ greater than $0$ such that $f^p(x)=x$ for all $x\in X$ is called the \textit{period} of $f$.

The following example shows that the pointwise periodic does not imply that $(f,X)$ is periodic in general.

\begin{exa}\label{R3.3.1}
Let $r_n, n=1,2,\dotsc$ be a sequence of rational numbers with $0<r_n<1$ such that $r_n\downarrow0$ as $n\to\infty$. Let $\mathbb{S}=\{z\,|\,z\in\mathbb{C},|z|=1\}$ and
$\mathbb{S}_n=\{z\,|\,z\in\mathbb{C}, |z|=1-r_n\}$ for $n=1,2,\dotsc$ are concentric circles in the complex plane. Define $f_n\colon\mathbb{S}_n\rightarrow\mathbb{S}_n$ by $z\mapsto e^{2\pi ir_n}z$, for $n=1,2,\dotsc$. Let $X=\mathbb{S}\cup(\bigcup_{n=1}^\infty\mathbb{S}_n)$ and define $f\colon X\rightarrow X$ by $f|\mathbb{S}_n=f_n$ and $f|\mathbb{S}=\textrm{id}_{\mathbb{S}}$. Clearly, the $\mathbb{Z}$-flow induced by the pointwise periodic cascade $(f,X)$ is pointwise (regularly almost) periodic but it is not equicontinuous at every point of $\mathbb{S}$. In fact, in view of Lemma~\ref{R2.4}, $(f,X)$ is not a locally weakly a.p $\mathbb{Z}$-flow.
\end{exa}

However, if the phase space $X$ is a locally euclidean space, then a pointwise periodic uniform transformation is periodic:

\begin{sthm}[{Montgomery~\cite{Mo}}]\label{R3.3.2}
If $f$ is a pointwise periodic uniform homeomorphism of a boundaryless $n$-manifold $M$ into itself, then $(f,M)$ is periodic.
\end{sthm}

A sharp case beyond Montgomery's Theorem is a compact metric $\mathbb{Z}$-flow with zero-dimensional phase space as follows:

\begin{sthm}\label{R3.3.3}
Let $X$ be a compact space with $\dim X=0$. If $f$ is a pointwise periodic homeomorphism of $X$ into itself then $(f,X)$ is regularly a.p but it is not necessarily periodic.
\end{sthm}

\begin{proof}
1). As a $\mathbb{Z}$-flow $(f,X)$ is distal so $(f,X)$ is (uniformly) equicontinuous by Theorem~\ref{R3.1.1}. Further by compactness of $X$ and Lemma~\ref{R1.5}d, it follows easily that $(f,X)$ is a regularly a.p cascade.

2). For the second part, let's consider a counterexample (due to the referee) as follows: Let $X=\prod_{n=2}^\infty\mathbb{Z}/n\mathbb{Z}$, and, let $f\colon X\rightarrow X$ be defined in the ways: for every $x\in X$, if $x=(0,0,0,\dotsc)$ then $fx=x$, otherwise $fx$ is given by adding $1$ to the digit after the first nonzero entry of $x$. Then $(f,X)$ is pointwise periodic but has infinite order.
The proof is completed.
\end{proof}

Therefore, if $(f,X)$ is a pointwise periodic cascade with $X$ a compact boundaryless $n$-manifold or $X$ a 0-dimensional compact space, then $f$ has the topological entropy zero without using the variational principle of entropy.

Notice that as shown by Example~\ref{R3.3.1} and Theorem~\ref{R3.3.3}, if $X$ is not a manifold, then a pointwise periodic homeomorphism of $X$ need not be periodic.

\section{The case when $G$ is a finitely generated group}\label{secR4}

Let $G$ be discrete and finitely generated by $\Gamma$ such that $e\in\Gamma=\Gamma^{-1}$, and let $(G,X)$ be a flow. We shall show that the $\Gamma$-recurrence of type I (Def.~\ref{R2.2}a) is consistent with \cite{AGW} in this case. See Lemma~\ref{R4.4} below.

\begin{se}\label{R4.1}
Let $B_r=\{g\in G\,|\,g\not=e, |g|\le r\}$, for all $r\ge1$. For $g\in G$ with $g\not=e$, write $K(g)=B_{|g|-1}\cdot g$.
\end{se}

\begin{se}\label{R4.2}
A subset $C$ of $G$ is an AGW-\textit{cone}, namely, a cone in the sense of Auslander, Glasner and Weiss~\cite{AGW}, if there is a sequence $g_n\in G$ with $|g_n|\to\infty$ such that for each $r\ge1$ there exists $n_r$
such that $B_r\cap K(g_n)$ is independent of $n$ for all $n\ge n_r$, and, $C=\lim_{n\to\infty}K(g_n)$. Since $G$ is discrete here, so $c\in C$ iff $c\in K(g_n)$ as $n$ sufficiently large. Moreover, $e\notin C$ and by the proof of Lemma~\ref{R2.1}c, it is easy to see that $C$ is thick in $G$.
\end{se}

\begin{se}[{cf.~\cite[Def.~1.6]{AGW}}]\label{R4.3}
We say that $x\in X$ is AGW-\textit{recurrent}, if $Cx\cap U\not=\emptyset$ for every $U\in\mathfrak{N}_x$ and every AGW-cone $C$ in $G$.
\end{se}

\begin{lem}\label{R4.4}
A point $x$ is AGW-recurrent under $(G,X)$ if and only if $x$ is $\Gamma$-recurrent of type I in the sense of Def.~\ref{R2.2}a.
\end{lem}

\begin{proof}
1). Let $\C$ be a $\Gamma$-cone in $G$ with $\K(g_i)=\Gamma^{|g_i|-1}\cdot g_i\rightharpoonup\C$, where $\{g_i\,|\,i\in\Lambda\}$ is a net in $G$ such that $|g_i|\nearrow\infty$.
For all integer $k\ge1$, there is an $i_k\in\Lambda$ such that $|g_i|\ge k$ for all $i\ge i_k$ and $|g_{i_k}|<|g_{i_{k+1}}|$.
Clearly $|g_{i_k}|\to\infty$ as $k\to\infty$. Moreover, the sequence $\{g_{i_k}\}_{k=1}^\infty$ is a subnet of $\{g_i\}$. Indeed, for $i^\prime\in\Lambda$, there is an integer $k\ge1$ with $|g_{i^\prime}|<k\le |g_{i_{k^\prime}}|$ for all $k^\prime>k$ so $i^\prime\le i_{k^\prime}$ for all $k^\prime>k$.
Therefore, using $\{i_k\}$ in place of $\Lambda$ if necessary, we can assume $\{g_i\}_{i=1}^\infty$ is a sequence in $G$ with $|g_i|\nearrow\infty$ as $i\to\infty$.

Each $B_r$ is finite, so we may choose a subsequence $\{i_k\}$ from $\Lambda$ so that each $K(g_{i_k})\cap B_r$ is eventually constant. Then by a diagonal process we can choose a subsequence (and relabel) $\{i_k\}_{k=1}^\infty$ from $\Lambda$ such that $K(g_{i_k})\to C$ as $k\to\infty$. Then $C$ is an AGW-cone in $G$ such that $C\subseteq\C$ by Lemma~\ref{R1.1}.

2). Let $C$ be a cone in $G$ with $C=\lim_iK(g_i)$ in the sense of Def.~\ref{R4.2}, where $\{g_i\}_{i=1}^\infty$ is a sequence in $G$ such that $|g_i|\to\infty$ and for each $r\ge1$ there exists an $n_r\ge1$ such that $B_r\cap K(g_i)$ is independent of $i$ for all $i\ge n_r$. By induction, we may require $n_r<n_{r+1}$ and $|g_{n_r}|<|g_{n_{r+1}}|$ for $r\ge1$. Then
$$
C={\bigcup}_{r=1}^\infty(B_r\cap K(g_{n_r})).
$$
Let $\K(g_{n_r})\rightharpoonup\C$ in the sense of Def.~\ref{R2.1}b. Let $c\in\C$. Then there exists a subnet $\{n_{r^\prime}\}$ from $\{n_r\}$ and $k_{n_{r^\prime}}\in\K(g_{n_{r^\prime}})$ such that $k_{n_{r^\prime}}\to c$. Since $G$ is discrete and $\{c\}$ is an open neighborhood of $c$ in $G$, we can assume $k_{n_{r^\prime}}=c$ for all $r^\prime$. Since we have for all $n_r$ that $n_{r^\prime}\ge n_r$ as $r^\prime$ sufficiently large, so $c\in C$ and $\C\subseteq C$.

The above discussion of 1) and 2) implies that $x$ is AGW-recurrent iff $x$ is also $\Gamma$-recurrent of type I. The proof is complete.
\end{proof}

\begin{prop}[{cf.~\cite[Thm.~1.8]{AGW}}]\label{5.5}
Let $(G,X)$ be a flow, where $G$ is finitely generated and $X$ is a 0-dimensional space. Then the following conditions are equivalent:
\begin{enumerate}[(i)]
\item $(G,X)$ is pointwise AGW-recurrent.
\item $(G,X)$ is pointwise a.p.
\item The $G$-orbit closure relation $R_o(X)$ is closed.
\end{enumerate}
\end{prop}

\begin{proof}
In view of Lemma~\ref{R4.4}, this is a special case of Theorem~\ref{R2.6}a.
\end{proof}

\begin{prop}[{cf.~\cite[Thm.~1]{G44}, \cite[Thm.~I(a)]{ES} and \cite[Thm.~1.4]{F81} for (2) by different approaches}]\label{R4.6}
Suppose $f\colon X\rightarrow X$ is a homeomorphism of $X$, which is thought of as a $\mathbb{Z}$-flow. Let $x\in X$. Then:
\begin{enumerate}[(1)]
\item $x$ is recurrent of type I (cf.~Def.~\ref{R2.2}a) iff $x$ is ``stable in the sense of Poisson'', i.e., there exists a net $\{i_n\}$ in $\mathbb{Z}$ with $i_n\to+\infty$ and a net $\{j_n\}$ in $\mathbb{Z}$ with $j_n\to-\infty$, such that $f^{i_n}x\to x$ and $f^{j_n}x\to x$ simultaneously.
\item If $\kappa\in\mathbb{N}$ and $i_n\in\mathbb{N}$ is a net with $i_n\to\infty$ such that $f^{i_n}x\to x$, then there is a net $k_n\in\kappa\mathbb{N}$ with $k_n\to\infty$ such that $f^{k_n}x\to x$.
\end{enumerate}
\end{prop}

\begin{proof}
(1). Note that a $\Gamma$-cone $\C$ in $\mathbb{Z}$ with $\Gamma=\{-1,0,1\}$ is either $\C=\mathbb{N}$ or $\C=-\mathbb{N}$, where $\mathbb{N}$ is the set of positive integers. Then the statement follows easily from this fact.

(2). Let $i_n=k_n-\tau_n$, where $k_n\in\kappa\mathbb{N}$ and $\tau_n\in[0,\kappa)\cap\mathbb{Z}$. We may assume $\tau_n\to\tau\in\mathbb{Z}_+$. Then $f^{k_n-\tau_n}x\to x$ so that
$f^{k_n}x\to f^\tau x$. Further $f^{k_n+\tau}x\to f^{2\tau}x$. We shall show that $f^{2\tau}x$ is a limit of the set $f^{\kappa\mathbb{N}}x=\{f^tx\,|\,t\in\kappa\mathbb{N}\}$. For this, let $\varepsilon,\alpha\in\mathscr{U}$ with $\alpha^2\subseteq\varepsilon$. Let $n_0$ be an index then there is an index $v>n_0$ such that $(f^{2\tau}x,f^{k_v+\tau}x)\in\alpha$. There is an index $\delta\in\mathscr{U}$ such that if $(q,f^\tau x)\in\delta$ then $(f^{k_v}q,f^{k_v+\tau}x)\in\alpha$. Then by $f^{k_n}x\to f^\tau x$ and $(f^{2\tau}x,f^{k_v+\tau}x)\in\alpha$, it follows that as $n$ sufficiently big, $(f^{k_n+k_v}x,f^{2\tau}x)\in\varepsilon$. Thus $f^{2\tau}x$ is a limit of $f^{\kappa\mathbb{N}}x$.

The same reasoning would show that $f^{3\tau}x, f^{4\tau}x, \dotsc$ are also limits of $f^{\kappa\mathbb{N}}x$. Then $f^{\kappa\tau}x$ and so $x$ is a limit of $f^{\kappa\mathbb{N}}x$. The proof is complete.
\end{proof}

Condition (3) in the following corollary is exactly the definition of the recurrence in \cite{GH} for generative groups.

\begin{cor*}
Let $f\colon X\rightarrow X$ be a homeomorphism of $X$ and $x\in X$. Then: (1) $x$ is Poisson stable iff (2) $f^Sx\cap U\not=\emptyset$ for all $U\in\mathfrak{N}_x$ and all subsemigroup $S$ of $\mathbb{Z}$ iff (3) $f^Sx\cap U\not=\emptyset$ for all $U\in\mathfrak{N}_x$ and all thick subsemigroup $S$ of $\mathbb{Z}$.
\end{cor*}

\begin{proof}
(1)$\Rightarrow$(2) follows obviously from Proposition~\ref{R4.6}. (2)$\Rightarrow$(3) is trivial. Finally (3)$\Rightarrow$(1) follows from that $\mathbb{N}$ and $-\mathbb{N}$ both are thick semigroups in $\mathbb{Z}$.
\end{proof}

There exists an example~\cite{AGW} of a $\mathbb{Z}$-flow where all points are positively recurrent, but there are points which are not negatively recurrent. Thus the recurrence of type I is stronger than the positive recurrence.

\begin{prop}[{cf.~\cite[Thm.~7.10]{GH} for $G$ abelian finitely generated}]\label{R4.7}
Let $(G,X)$ be a flow with $G$ finitely generated and with $X$ 0-dimensional compact. If $(G,X)$ is pointwise regularly a.p, then $(G,X)$ is regularly a.p.
\end{prop}

\begin{proof}
By Lemma~\ref{R1.5}d, $(G,X)$ is distal. Indeed, let $x_1,x_2\in X$, $U\in\mathfrak{N}_{x_1}$, and $V\in\mathfrak{N}_{x_2}$. There are normal syndetic subgroups $H_1$ and $H_2$ of $G$ with $H_1x_1\subseteq U$ and $H_2x_2\subseteq V$. By Lemma~\ref{R1.5}d, there is a normal syndetic subgroup $A$ of $G$ with $A\subseteq H_1\cap H_2$. Thus
$A(x_1,x_2)\subseteq U\times V$. Then $(G,X\times X)$ is pointwise a.p and this shows that $(G,X)$ is distal.

Further by Theorem~\ref{R3.1.1}, it follows that $(G,X)$ is equicontinuous. Finally by \cite[Remark~5.02, Theorem~5.17]{GH}, we see that $(G,X)$ is regularly a.p. The proof is complete.
\end{proof}

\begin{prop}[{cf.~\cite[Thm.~7.11]{GH} for $G$ abelian finitely generated}]\label{R4.8}
Let $(G,X)$ be a flow with $G$ finitely generated and with $X$ a compact metric space. Then $(G,X)$ is regularly a.p if and only if $(G,X)$ is a pointwise regularly a.p and weakly a.p flow.
\end{prop}

\begin{proof}
The proof of \cite[Thm.~7.11]{GH} is still valid for this case using Theorem~\ref{R2.6}a in place of \cite[Theorem~7.08]{GH}. In fact, the necessity  is trivial. To prove the sufficiency, let $S$ be any normal syndetic subgroup of $G$. By \cite[Theorems~5.17 and 4.24]{GH}, it is enough to show that $\mathscr{O}_S\colon x\in X\mapsto\overline{Sx}\in 2^X$ is continuous. Note that by Inheritance Theorem (cf.~Lemma~\ref{R1.5}e), $(S,X)$ is pointwise a.p so $\{\overline{Sx}\,|\,x\in X\}$ is a partition of $X$. Let there be a sequence $x_n\to x_0$ in $X$. Then by Lemma~\ref{R2.4}, $\overline{Gx_n}\to\overline{Gx_0}$ in $2^X$ and further $Y=\bigcup_{i=0}^\infty\overline{Gx_i}$ is closed in $X$. By \cite[Theorem~5.08]{GH}, $\dim \overline{Gx}=0$ for all $x\in X$. Since a separable metric space that is countable union of closed zero-dimensional subsets is itself zero-dimensional (cf.~\cite[Theorem~1.3.1]{En}), $Y$ is $G$-invariant such that $\dim Y=0$. Then $(S,Y)$ is locally weakly a.p by Theorem~\ref{R2.6}a and Corollary~\ref{R1.8}b. Thus $\overline{Sx_n}\to\overline{Sx_0}$. The proof is completed.
\end{proof}

\begin{prop}\label{R4.9}
Let $(G,X)$ be a flow with $G$ finitely generated. Let $U$ be a compact open subset of $X$, which consists of a.p points. Then there is a finite set $K$ in $G$ such that $GU=KU$.
\end{prop}

\begin{proof}
Let $G$ be finitely generated by $\Gamma$ with $\Gamma=\Gamma^{-1}$. For $\ell\in\mathbb{N}$ set $W_\ell=\Gamma^\ell U$. If there exists $\ell\in\mathbb{N}$ such that $\Gamma W_\ell\subseteq W_\ell$, then by induction $GW_\ell\subseteq W_\ell$. Hence in particular, $GU\subseteq \Gamma^\ell U$, and this completes the proof by letting $K=\Gamma^\ell$. Now assume that for all $\ell\in\mathbb{N}$ we have $\Gamma W_\ell\nsubseteq W_\ell$; i.e., $\Gamma^{\ell+1}U\nsubseteq\Gamma^\ell U$. Then for every $\ell\in\mathbb{N}$ there are $t_\ell\in\Gamma^{\ell+1}$ with $|t_\ell|=\ell+1$ and $y_\ell\in U$ such that $t_\ell y_\ell\notin \Gamma^\ell U$. We may assume (a subnet of) $y_\ell\to y\in U$. Let $s\in N_G(y,U)$; then $sy_\ell\to sy\in U$, so as $\ell$ sufficiently large, $sy_\ell\in U$, hence $t_\ell y_\ell\in t_\ell s^{-1}U$, and therefore $t_\ell s^{-1}\notin \Gamma^\ell$ and $s^{-1}\notin t_\ell^{-1}\Gamma^\ell$, or equivalently, $s\notin \Gamma^\ell t_\ell$. Let $\K(t_\ell)\rightharpoonup\C$, a $\Gamma$-cone in $G$ as in \ref{R2.1}b. Then $s\notin\C$ and $\C\cap N_G(y,U)=\emptyset$, contrary to that $\C$ is thick and $N_G(y,U)$ is syndetic in $G$. The proof is complete.
\end{proof}

\begin{rems}\label{R4.10}
\begin{enumerate}[(1)]
\item The $\Gamma$-recurrence of type I (Def.~\ref{R2.2}a) is conceptually dependent of the generating set $\Gamma$ of $G$. So it should be interested to generalize Proposition~\ref{R4.6} from $G=\mathbb{Z}$ to a more general non-abelian case.
\item Can we remove the ``metric'' condition in Proposition~\ref{R4.8}?
\item Can the assumption that $G$ is finitely generated be replaced by that $G$ is compactly generated in any of Propositions~\ref{R4.8} and \ref{R4.9}?
\end{enumerate}
\end{rems}
\section{The case when $G$ is equicontinuously generated}\label{secR5}

Let $\mathcal{C}_u(X,X)$ be the space of continuous maps from $X$ to itself with the topology of uniform convergence on compacta. Clearly, $(\textrm{Homeo}\,(X),X)$ defined by $\textrm{Homeo}\,(X)\times X\xrightarrow{(f,x)\mapsto fx}X$ is a flow, where $(f,x)\mapsto fx$ is jointly continuous but $\overline{\textrm{Homeo}\,(X)x}$ need not be compact for $x\in X$. In the sequel let $S\subset\textrm{Homeo}\,(X)$ such that:
\begin{enumerate}[1)]
\item $e=\textrm{id}_X\in S=S^{-1}$ and
\item $S$ acting equicontinuously on $X$ (i.e., given $\varepsilon\in\mathscr{U}$ and $x_0\in X$, there is a $U\in\mathfrak{N}_{x_0}$ such that $(sx,sx_0)\in\varepsilon$ for all $x\in U$ and all $s\in S$; cf.~Def.~\ref{R1.2}d).
\end{enumerate}

\begin{se}\label{R5.1}
Set $\langle S\rangle=\bigcup_{r=1}^\infty S^r$. Clearly $\langle S\rangle$ is a subgroup of $\textrm{Homeo}\,(X)$, which is said to be \textit{equicontinuously generated by $S$}.
\end{se}

\begin{se}\label{R5.2}
Write $\Gamma=\textrm{cls}_uS$, where $\textrm{cls}_u$ denotes the closure relative to $\mathcal{C}_u(X,X)$. Then:

\begin{lem*}
If $Sx$ is relatively compact in $X$ for all $x\in X$, then $\Gamma$ is a compact subset of $\mathcal{C}_u(X,X)$ and $\Gamma\subset\textrm{Homeo}\,(X)$ such that $e\in\Gamma=\Gamma^{-1}$.
\end{lem*}

\begin{proof}
Since $S$ is equicontinuous, so is $\Gamma$. Then $\Gamma$ is compact in $\mathcal{C}_u(X,X)$ by Ascoli's Theorem (cf.~\cite[Theorem~7.17]{K}). Moreover, $\Gamma x$ is compact in $X$. We need prove that $\Gamma\subset\textrm{Homeo}\,(X)$. For this, let $S\ni f_n\to f\in\Gamma$.
Let $x\not=y$ and $\varepsilon\in\mathscr{U}$ with $(x,y)\notin\varepsilon$. If $fx=fy=z$, then for $\delta\in\mathscr{U}$ we have $x, y\in f_n^{-1}\delta[z]$ as $n$ sufficiently large, contrary to equicontinuity of $S$. Thus $f$ is injective.
Let $x\in X$. Since $f_nX=X$, there is a point $x_n\in \Gamma x$ with $f_nx_n=x$ or $x_n=f_n^{-1}x$. We can assume (a subnet of) $x_n\to x^\prime$ for $\Gamma x$ is compact. Then $fx^\prime=x$ so $fX=X$. Thus $f$ is a continuous bijection. Now let (a subnet of) $f_n^{-1}\in S\to g\in\Gamma$. Since $e=f_nf_n^{-1}\to fg$, so $fg=e$, $g=f^{-1}\in\Gamma$, and  $f\in\textrm{Homeo}\,(X)$. The proof is complete.
\end{proof}
\end{se}

\begin{se}\label{R5.3}
Let $Sx$ be relatively compact in $X$ for all $x\in X$. Set $\langle\Gamma\rangle=\bigcup_{r=1}^\infty \Gamma^r$. Then $\langle\Gamma\rangle$ is a compactly generated subgroup of $\textrm{Homeo}\,(X)$ with a generating set $\Gamma$. Moreover, $(\langle\Gamma\rangle,X)$ is a flow; however, $\overline{\langle\Gamma\rangle x}$ need not be compact for $x\in X$. It is evident that $S^r$ is dense in $\Gamma^r$ for all $r\ge1$ and $\langle S\rangle$ is dense in $\langle\Gamma\rangle$ under the topology of uniform convergence.
\end{se}

Following Def.~\ref{R2.2}b, a point $x\in X$ is $S$-recurrent of type II under $(\langle S\rangle,X)$ if for every $U\in\mathfrak{N}_x$ and every net $\{g_i\}$ in $\langle S\rangle$ with $|g_i|\to\infty$, there is an integer $n$, a subnet $\{g_{i_j}\}$ from $\{g_i\}$ and $c_j\in S^{|g_{i_j}|-1}\cdot g_{i_j}$ such that $c_jx\in U$ and $|c_j|\le n$.

We do not know if a point $S$-recurrent of type II is $\Gamma$-recurrent of type II, yet it is $\Gamma$-recurrent of type I by Lemma~\ref{R5.4} below.

\begin{lem}\label{R5.4}
Let $Sx$ be relatively compact in $X$ for all $x\in X$. If $x_0$ is $S$-recurrent of type II under $(\langle S\rangle,X)$, then $x_0$ is $\Gamma$-recurrent of type I under $(\langle\Gamma\rangle,X)$.
\end{lem}

\begin{proof}
Let $\{g_i\,|\,i\in\Lambda\}$ be a net in $\langle\Gamma\rangle$ with $|g_i|\nearrow\infty$ and $\K(g_i)\rightharpoonup\C$. Let $U, V\in\mathfrak{N}_{x_0}$ with $\overline{V}\subseteq U$. We need prove that $\C x_0\cap U\not=\emptyset$. Since $\langle S\rangle$ is dense in $\langle\Gamma\rangle$, we can assume $g_i\in\langle S\rangle$ for all $i\in\Lambda$. Further by definition, there is an integer $n\ge1$, a subnet $\{g_{i_j}\}$ from $\{g_i\}$ and $c_j\in S^{|g_{i_j}|-1}\cdot g_{i_j}\subseteq\K(g_{i_j})$ such that $c_jx_0\in V$ and $|c_j|\le n$. By $c_j\in\Gamma^n$ and $\Gamma^n$ is compact, it follows that (a subnet of) $c_j\to c\in\C$ and $cx_0\in U$. The proof is complete.
\end{proof}

\begin{prop}[{cf.~\cite[Thm.~2.2]{R1} and \cite[Thm.~1.2]{R2} for $X$ s.t. $\dim X=0$}]\label{R5.5}
Let $Sx$ be relatively compact in $X$ for all $x\in X$. Let $G$ be an equicontinuously generated group by $S$. Let $V$ be an open subset of $X$. Suppose $V_G^\infty$ is compact. Then $V_G^\infty$ is open if and only if $\overline{Gx}\cap V_G^\infty=\emptyset$ for every $x\in X\setminus V_G^\infty$.
\end{prop}

\begin{proof}
Let $G=\langle S\rangle$, where $S$ is the equicontinuously generating set of $G$. Then $\overline{Gx}=\overline{\langle\Gamma\rangle x}$ for all $x\in X$. Since $V_G^\infty$ is closed, so $V_G^\infty=\bigcap_{g\in\langle\Gamma\rangle}gV=V_{\langle\Gamma\rangle}^\infty$. Then Proposition~\ref{R5.5} follows from Theorem~\ref{R2.6}b. The proof is complete.
\end{proof}

\begin{prop}[{cf.~\cite[Thm.~1.3]{R2} using Prop.~\ref{R5.5}}]\label{R5.6}
Let $\dim X=0$ and $G=\langle S\rangle$ such that $\overline{Gx}$ is compact for all $x\in X$. Then the following are pairwise equivalent:
\begin{enumerate}[i)]
\item $(G,X)$ is pointwise $S$-recurrent of type II.
\item $(G,X)$ is pointwise a.p.
\item Given any compact open subset $V$ of $X$, $V_G^*$ is open.
\item The orbit closure relation $R_o(X)$ of $(G,X)$ is closed.
\item There exists an extension $\rho\colon(G,X)\rightarrow(G,Y)$ such that:
\begin{enumerate}[1)]
\item $Y$ is locally compact 0-dimensional,
\item $gy=y$ for all $y\in Y$ and $g\in G$, and
\item $\rho^{-1}y$ is a minimal subset of $(G,X)$ for each $y\in Y$.
\end{enumerate}
\end{enumerate}
\end{prop}

\begin{proof}
By Lemma~\ref{R5.4}, $(\langle\Gamma\rangle,X)$ is pointwise $\Gamma$-recurrent of type I.
Since $\overline{Gx}=\overline{\langle \Gamma\rangle x}$ for all $x\in X$, $x$ is a.p under $(G,X)$ iff $x$ is a.p under $(\langle \Gamma\rangle,X)$. Then
i)\,$\Leftrightarrow$\,ii)\,$\Leftrightarrow$\,iii)\,$\Leftrightarrow$\,iv)\,$\Leftrightarrow$\,iv$^\prime$) by Theorem~\ref{R2.6}a, where
\begin{enumerate}[iv$^\prime$)]
\item $\mathscr{O}_G\colon X\rightarrow 2^X$ is upper semi-continuous.
\end{enumerate}

v)$\Rightarrow$ii): By 3) of v).

iv$^\prime$)$\Rightarrow$v): Assume iv$^\prime$). Define $Y=X/G=X/R_o(X)$ and let $\rho\colon X\rightarrow Y$ be the canonical quotient mapping. Then $\rho$ is closed (cf.~\cite[Theorem~3.12]{K}), and moreover, by ii) and Lemma~\ref{R1.4}b it is easy to see that $\rho$ is open so $Y$ is Hausdorff. Since $\rho$ is clopen and $X$ is locally compact 0-dimensional, $Y$ is locally compact 0-dimensional. Thus $(G,Y)$ is a flow having the properties 1), 2) and 3).

The proof of Proposition~\ref{R5.6} is thus completed.
\end{proof}
\medskip

\noindent
\textbf{Acknowledgments. }
The author would like to thank the referee for her/his constructive comments that improve the manuscript.
This work was supported by National Natural Science Foundation of China (Grant No. 11790274) and PAPD of Jiangsu Higher Education Institutions.

\end{document}